\documentclass[a4paper, leqno]{article}
\usepackage[utf8]{inputenc}
\usepackage[french,english]{babel}
\usepackage[T1]{fontenc}
\usepackage{amsmath, amsfonts, amssymb, amsthm, dsfont}
\usepackage{a4wide}
\usepackage{soulutf8}
\usepackage{authblk}
\usepackage{color, soul}

\usepackage{tikz}
\usepackage{verbatim}
\usetikzlibrary{arrows}

\def\R{\mathbb{R}}

\def\N{\mathbb{N}}

\usepackage[titletoc,toc,title]{appendix}

\newtheorem{theorem}{Theorem}[section]
\newtheorem{remark}{Remark}[section]
\newtheorem{proposition}{Proposition}[section]
\newtheorem{lemma}{Lemma}[section]

\newtheorem*{notation}{Notation}
\newtheorem{definition}{Definition}[section]

\makeatletter

\@addtoreset{equation}{section}
\makeatother

\usepackage{enumitem} 
\usepackage[colorlinks=true]{hyperref}
\setlist[description]{style=nextline,labelwidth=0pt,leftmargin=30pt,itemindent=\dimexpr-20pt-\labelsep\relax} 

\makeatletter
\def\namedlabel#1#2{\begingroup
    #2%
    \def\@currentlabel{#2}%
    \phantomsection\label{#1}\endgroup
}
\makeatother

\usepackage{xcolor}

\hypersetup{ 
colorlinks=false, 
linkbordercolor={1 1 1}, 
citebordercolor={1 1 1} 
} 

\usepackage{float}
\usepackage{graphicx}
\usepackage{rotating}
\usepackage{epic}
\usepackage{eepic}
\usepackage{epsfig}
\usepackage{array, multirow, hhline, tabularx, graphicx, wrapfig}
\usepackage{tikz}

\everymath{\displaystyle}
\usepackage{color}
\usepackage{verbatim}
\usepackage{listings}
\usepackage[notlof, notlot, nottoc]{tocbibind}

\usepackage{enumitem}

\makeatletter
\newcommand\varitem[1]{\item[(H\arabic{enumi}\rlap{$#1$} )]%
  \edef\@currentlabel{(H\arabic{enumi}{$#1$})}}
\makeatother

\begin{document}

\renewcommand{\labelitemi}{$\bullet$}


\newcommand{\lang}{\begin{picture}(5,7)
\put(1.1,2.5){\rotatebox{45}{\line(1,0){6.0}}}
\put(1.1,2.5){\rotatebox{315}{\line(1,0){6.0}}}
\end{picture}}
\newcommand{\rang}{\begin{picture}(5,7)
\put(.1,2.5){\rotatebox{135}{\line(1,0){6.0}}}
\put(.1,2.5){\rotatebox{225}{\line(1,0){6.0}}}
\end{picture}}

\title{Optimal control of infinite-dimensional piecewise deterministic Markov processes and application to the control of neuronal dynamics via Optogenetics}

\date{}

\author{Vincent Renault, Michèle Thieullen\thanks{Sorbonne Universit\'es, UPMC Univ Paris 06, CNRS UMR 7599, Laboratoire de Probabilit\'es et Mod\`eles Al\'eatoires, F-75005, Paris, France. email: \href{mailto:vincent.renault@upmc.fr}{vincent.renault@upmc.fr}, \href{mailto:michele.thieullen@upmc.fr}{michele.thieullen@upmc.fr}}, Emmanuel Trélat\thanks{Sorbonne Universit\'es, UPMC Univ Paris 06, CNRS UMR 7598, Laboratoire Jacques-Louis Lions, Institut Universitaire de France, F-75005, Paris, France. email: \href{mailto:emmanuel.trelat@upmc.fr}{emmanuel.trelat@upmc.fr}}}


\maketitle

\begin{abstract}

In this paper we define an infinite-dimensional controlled piecewise deterministic Markov process (PDMP) and we study an optimal control problem with finite time horizon and unbounded cost. This process is a coupling between a continuous time Markov Chain and a set of semilinear parabolic partial differential equations, both processes depending on the control. We apply dynamic programming to the embedded Markov decision process to obtain existence of optimal relaxed controls and we give some sufficient conditions ensuring the existence of an optimal ordinary control. This study, which constitutes an extension of controlled PDMPs to infinite dimension,  is motivated by the control that provides Optogenetics on neuron models such as the Hodgkin-Huxley model. We define an infinite-dimensional controlled Hodgkin-Huxley model as an infinite-dimensional controlled piecewise deterministic Markov process and apply the previous results to prove the existence of optimal ordinary controls for a tracking problem.



\end{abstract}


\section*{Introduction}

Optogenetics is a recent and innovative technique which allows to induce or prevent electric shocks in living tissues, by means of light stimulation. Successfully demonstrated in mammalian neurons in 2005 (\cite{millisecondTimescale}), the technique relies on the genetic modification of cells to make them express particular ionic channels, called rhodopsins, whose opening and closing are directly triggered by light stimulation. One of these rhodopsins comes from an unicellular flagellate algae, \textit{Chlamydomonas reinhardtii}, and has been baptized Channelrodhopsins-2 (ChR2). It is a cation channel that opens when illuminated with blue light. 

Since the field of Optogenetics is young, the mathematical modeling of the phenomenon is quite scarce. Some models have been proposed, based on the study of the photocycles initiated by the absorption of a photon. In 2009, Nikolic and al. \cite{Nikolic} proposed two models for the ChR2 that are able to reproduce the photocurrents generated by the light stimulation of the channel. Those models are constituted of several states that can be either conductive (the channel is open) or non-conductive (the channel is closed). Transitions between those states are spontaneous, depend on the membrane potential or are triggered by the absorption of a photon.
For example, the four-states model of Nikolic and al. \cite{Nikolic} has two open states ($o_1$ and $o_2$) and two closed states ($c_1$ and $c_2$). Its transitions are represented on Figure \ref{ChR2figure} 

\begin{figure}[!h]
\begin{center}
\begin{tikzpicture}[->,>=stealth',shorten >=1pt,auto,node distance=3cm,
  thick,main node/.style={circle,draw,font=\sffamily\Large\bfseries}]

  \node[main node] (1) {$o_1$};
  \node[main node] (2) [right of=1] {$o_2$};
  \node[main node] (3) [below of=2] {$c_2$};
  \node[main node] (4) [below of=1] {$c_1$};

  \path[every node/.style={font=\sffamily\small}]
    (1) edge [bend left] node [right] {$K_{d1}$} (4)
        edge [bend right] node [above] {$e_{12}$} (2)
    (2) edge [bend right] node [above] {$e_{21}$} (1)
        edge [bend right] node[right] {$K_{d2}$} (3)
    (3) edge [bend right] node [right] {$\varepsilon_2 u(t)$} (2)
        edge node {$K_r$} (4)
    (4)edge [bend left] node[left] {$\varepsilon_1 u(t)$} (1);

   \draw [dashed] (-2,-2) -- (-1,-2) node [below left] {$light$};
    \draw [dashed] (5,-2) -- (4,-2) node [below right] {$light$} ;

\end{tikzpicture}
\caption{Simplified four states ChR2 channel : $\varepsilon_1$, $\varepsilon_2$, $e_{12}$, $e_{21}$, $K_{d1}$, $K_{d2}$ and $K_r$ are positive constants.}\label{ChR2figure}
\end{center}
\end{figure}
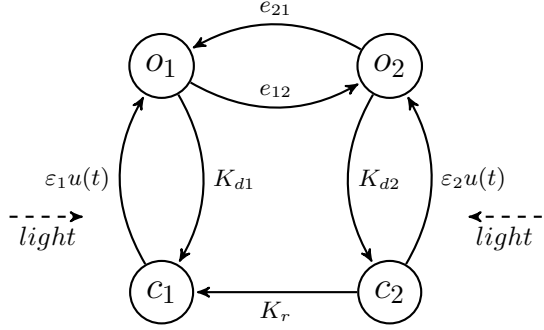

The purpose of this paper is to extend to infinite dimension the optimal control of \textit{Piecewise Deterministic Markov Processes} (PDMPs) and to define an infinite-dimensional controlled Hodgkin-Huxley model, containing ChR2 channels, as an infinite-dimensional controlled PDMP and prove existence of optimal ordinary controls. We now give the definition of the model. 

We consider an axon, described as a 1-dimensional cable and we set $I=[0,1]$ (the more physical case $I=[-l,l]$ with $2l>0$ the length of the axon is included here by a scaling argument).   Let $D_{ChR2} := \{o_1,o_2,c_1,c_2\}$. Individually, a ChR2 features a stochastic evolution which can be properly described by a Markov Chain on the finite space constituted of the different states that the ChR2 can occupy. In the four-states model above, two of the transitions are triggered by light stimulation, in the form of a parameter $u$ that can evolve in time. Here $u(t)$ is physically proportional to the intensity of the light with which the protein is illuminated. For now, we will consider that when the control is on (\textit{i.e.}, when the light is on), the entire axon is uniformly illuminated. Hence for all $t\geq 0$, $u(t)$ features no spatial dependency. 

The deterministic Hodgkin-Huxley model was introduced in \cite{HH}. A stochastic infinite-dimensional model was studied in \cite{Austin}, \cite{Riedler}, \cite{GenadotAverage} and \cite{WainribRiedler}. The Sodium ($Na^+$) channels and Potassium ($K^+$) channels are described by two pure jump processes with state spaces $D_1:=\{n_0,n_1,n_2,n_3,n_4\}$ and \\ $D_2:=\{m_0h_1,m_1h_1,m_2h_1,m_3h_1,m_0h_0,m_1h_0,m_2h_0,m_3h_0\}.$

For a given scale $N\in \mathbb{N^*}$, we consider that the axon is populated by $N_{hh}=N-1$ channels of type $Na^+$, $K^+$ or $ChR2$, at positions $\frac{1}{N}(\mathbb{Z}\cap N \mathring{I})$. In the sequel we will use the notation $I_N:=\mathbb{Z}\cap N \mathring{I}$. We consider the Gelfand triple ($V,H,V^*$) with $V := H_0^1(I)$ and $H :=L^2(I)$. The process we study is defined as a controlled infinite-dimensional \textit{Piecewise Deterministic Markov Process} (PDMP). All constants and auxiliary functions in the next definition will be defined further in the paper.

\begin{definition}\label{stoInfiniteHH} \textbf{Stochastic controlled infinite-dimensional Hodgkin-Huxley-ChR2 model}.
Let $N\in \mathbb{N^*}$. We call $N^{\mathrm{th}}$ stochastic controlled infinite-dimensional Hodgkin-Huxley-ChR2 model the controlled PDMP $(v(t),d(t))\in V\times D_N$ defined by the following characteristics: 
\begin{itemize}

\item A state space $V\times D_N$ with $D_N = D^{I_N}$ and $D= D_1\cup D_2 \cup D_{ChR2}$.

\item A control space $U=[0,u_{max}]$, $u_{max}>0$.

\item A set of uncontrolled PDEs:
For every $d\in D_N$,
\begin{equation}\label{PDEHHsto}
\left\{
\begin{aligned}
v'(t) &= \frac{1}{C_m}\Delta v(t) + f_d(v(t)),\\
v(0) &= v_0 \in V, \quad v_0(x) \in [V_-,V_+] \quad \forall x\in I,\\
v(t,0) &=v(t,1)=0, \quad \forall t > 0,
\end{aligned}
\right.
\end{equation}

with 

\begin{align}
 \mathcal{D}(\Delta) &= V,\nonumber\\
f_d(v) &:= \frac{1}{N}\sum_{i\in I_N} \Big( g_K\mathbf{1}_{\{d_i=n_4\}} (V_K-v(\frac{i}{N})) + g_{Na} \mathbf{1}_{\{d_i=m_3h_1\}}(V_{Na}-v(\frac{i}{N}))\label{fdHH} \\
& \qquad+ g_{ChR2}( \mathbf{1}_{\{d_i=O_1\}}+\rho  \mathbf{1}_{\{d_i=O_2\}})(V_{ChR2}-v(\frac{i}{N}))+ g_L(V_L-v(\frac{i}{N}))\Big) \delta_{\frac{i}{N}}\nonumber,
\end{align}

with $\delta_z\in V^*$ the Dirac mass at $z\in I$. 
\item A jump rate function $\lambda : V\times D_N\times U\rightarrow \mathbb{R}_+$ defined for all $(v,d,u)\in H\times D_N\times U$ by 

\begin{equation}
\lambda_d(v,u) = \sum_{i\in I_N} \sum_{x\in D}\sum_{\substack{y\in D,\\ y\neq x}} \sigma_{x,y}(v(\frac{i}{N}),u)\mathbf{1}_{\{d_i=x\}}, 
\end{equation}

with $\sigma_{x,y} :\mathbb{R}\times U\rightarrow \mathbb{R}_+^{*}$  smooth functions for all $(x,y)\in D^2$. See Table \ref{jumprates} in Section \ref{ProofTheoSubsec} for the expression of those functions. 

\item A discrete transition measure $\mathcal{Q}:V\times D_N\times U\rightarrow \mathcal{P}(D_N)$ defined for all  $(v,d,u)\in E\times D_N\times U$ and $y\in D$  by 

\begin{equation}
\mathcal{Q}(\{d^{i:y}\}|v,d) = \frac{\sigma_{d_i,y}(v(\frac{i}{N}),u)\mathbf{1}_{\{d_i\neq y\}}}{\lambda_d(v,u)},
\end{equation}

where $d^{i:y}$ is obtained from $d$ by putting its $i^{\mathrm{th}}$ component equal to $y$.
\end{itemize}

\end{definition}

From a biological point of view, the optimal control problem consists in mimicking an output signal that encodes a given biological behavior, while minimizing the intensity of the light applied to the neuron. For example, it can be a time-constant signal and in this case, we want to change the resting potential of the neuron to study its role on its general behavior. We can also think of pathological behaviors that would be fixed in this way. The minimization of light intensity is crucial because the range of intensity experimentally reachable is quite small and is always a matter of preoccupation for experimenters. These considerations lead us to formulate the following mathematical optimal control problem.

Suppose we are given a \textit{reference signal} $V_{ref}\in V$. The control problem is then to find $\alpha \in \mathcal{A}$ that minimizes the following expected cost 

\begin{equation}\label{HHExpectedcostFunction}
J_z(\alpha) = \mathbb{E}_z^{\alpha}\left[\int_0^T \left(\kappa||X_t^{\alpha}(\phi)-V_{ref}||_V^2+\alpha(X_t^{\alpha})\right)\mathrm{d}t \right], \quad z\in\Upsilon,
\end{equation}

where $\mathcal{A}$ is the space of control strategies, $\Upsilon$ an auxiliary state space that comprises $V\times D_N$, $X_{\cdot}^{\alpha}$ is the controlled PDMP and $X_{\cdot}^{\alpha}(\phi)$ its continuous component.

We will prove the following result.

\begin{theorem}\label{ChR2TheoremExistenceOptimalControl}
Under the assumptions of Section \ref{EnlargedSpace}, there exists an optimal control strategy $\alpha^*\in\mathcal{A}$ such that for all $z\in\Upsilon$,
\[
J_z(\alpha^*) = \inf_{\alpha\in\mathcal{A}} \mathbb{E}_z^{\alpha}\left[\int_0^T \left(\kappa||X_t^{\alpha}(\phi)-V_{ref}||_V^2+\alpha(X_t^{\alpha})\right)\mathrm{d}t \right],
\]

and the value function $z\rightarrow \inf_{\alpha\in\mathcal{A}} J_z(\alpha)$ is continuous on $\Upsilon$.
	
\end{theorem}

Piecewise Deterministic Markov Processes constitute a large class of Markov processes suited to describe a tremendous variety of phenomena such as the behavior of excitable cells (\cite{Austin},\cite{Riedler},\cite{WainribReduction}), the evolution of stocks in financial markets (\cite{RiederAppli}) or the congestion of communication networks (\cite{Dumas}), among many others. PDMPs can basically describe any non diffusive Markovian system. The general theory of PDMPs, and the tools to study them, were introduced by Davis (\cite{Davis84}) in 1984, at a time when the theory of diffusion was already amply developed. Since then, they have been widely investigated in terms of asymptotic behavior, control, limit theorems and CLT, numerical methods, among others (see for instance \cite{Brandejsky}, \cite{CostaDufourStability}, \cite{CostaDufourSingular}, \cite{Crudu} and references therein). PDMPs are jump processes coupled with a deterministic evolution between the jumps. They are fully described by three local characteristics: the deterministic flow $\phi$, the jump rate $\lambda$, and the transition measure $\mathcal{Q}$. In \cite{Davis84}, the temporal evolution of a PDMP between jumps (i.e. the flow $\phi$) is governed by an Ordinary Differential Equation (ODE). For that matter, this kind of PDMPs will be referred to as finite-dimensional in the sequel.

Optimal control of such processes have been introduced by Vermes (\cite{Vermes}) in finite dimension. In \cite{Vermes}, the class of \textit{piecewise open-loop} controls is introduced as the proper class to consider to obtain strongly Markovian processes. A Hamilton-Jabobi-Bellman equation is formulated and necessary and sufficient conditions are given for the existence of optimal controls. The standard broader class of so-called \textit{relaxed} controls is considered and it plays a crucial role in getting the existence of optimal controls when no convexity assumption is imposed. This class of controls has been studied, in the finite-dimensional case, by Gamkrelidze (\cite{Gamkrelidze}),  Warga (\cite{Warga62A} and \cite{Warga62B}) and Young (\cite{YoungBook}). Relaxed controls provide a compact class that is adequate for studying optimization problems. Still in finite dimension, many control problems have been formulated and studied such as optimal control (\cite{Forwick}), optimal stopping (\cite{CostaOptimalStop}) or controllability (\cite{GoreacControllability}). In infinite dimension, relaxed controls were introduced by Ahmed (\cite{Ahmed83}, \cite{Ahmed78}, \cite{Ahmed93}). They were also studied by Papageorgiou in \cite{Papageorgiou} where the author shows the strong continuity of relaxed trajectories with respect to the relaxed control. This continuity result will be of great interest in this paper.

A formal infinite-dimensional PDMP was defined in \cite{Riedler} for the first time, the set of ODEs being replaced by a special set of Partial Differential Equations (PDE). The extended generator and its domain are provided and the model is used to define a stochastic spatial Hodgkin-Huxley model of neuron dynamics. The optimal control problem we have in mind here regards those Hodgkin-Huxley type models. Seminal work on an uncontrolled infinite-dimensional Hodgkin-Huxley model was conducted in \cite{Austin} where the trajectory of the infinite-dimensional stochastic system is shown to converge to the deterministic one, in probability. This type of model has then been studied in \cite{WainribRiedler} in terms of limit theorems and in \cite{GenadotAverage} in terms of averaging. The extension to infinite dimension heavily relies on the fact that semilinear parabolic equations can be interpreted as ODEs in Hilbert spaces.

\medskip

To give a sense to Definition \ref{stoInfiniteHH} and to Theorem \ref{ChR2TheoremExistenceOptimalControl}, we will define a controlled infinite-dimensional PDMP for which the control acts on the three local characteristics. We consider controlled semilinear parabolic PDEs, jump rates $\lambda$ and transition measures $\mathcal{Q}$ depending on the control. This kind of PDE takes the form

\[
\dot{x}(t) = Lx(t) + f(x(t),u(t)),
\]
where $L$ is the infinitesimal generator of a strongly continuous semigroup and $f$ is some function (possibly nonlinear). The optimal control problem we address is the finite-time minimization of an unbounded expected cost functional along the trajectory of the form

\[
\min_{u} \mathbb{E} \int_0^T c(x(t),u(t))\mathrm{d}t,
\]
where $x(\cdot)$ is the continuous component of the PDMP, $u(\cdot)$ the control and $T>0$ the finite time horizon, the cost function $c(\cdot,\cdot)$ being potentially unbounded.

To address this optimal control problem, we use the fairly widespread approach that consists in studying the imbedded discrete-time Markov chain composed of the times and the locations of the jumps. Since the evolution between jumps is deterministic, there exists a one-to-one correspondence between the PDMP and a pure jump process that enable to define the imbedded Markov chain. The discrete-time Markov chain belongs to the class of \textit{Markov Decision Processes} (MDPs). This kind of approach has been used in \cite{Forwick} and \cite{Rieder} (see also the book \cite{MDPBook} for a self-contained presentation of MDPs). In these articles, the authors apply dynamic programming to the MDP derived from a PDMP, to prove the existence of optimal relaxed strategies. Some sufficient conditions are also given to get non-relaxed, also called ordinary, optimal strategies. However, in both articles, the PDMP is finite dimensional. To the best of our knowledge, the optimal control of infinite-dimensional PDMPs has not yet been treated and this is one of our main objectives here, along with its motivation, derived from the Optogenetics, to formulate and study infinite-dimensional controlled neuron models.

\medskip

The paper is structured as follows. In Section \ref{DefSection} we adapt the definition of a standard infinite-dimensional PDMP given in \cite{Riedler} in order to address control problems of such processes. To obtain a strongly Markovian process, we enlarge the state space and we prove an extension to controlled PDMPs of \cite[Theorem 4]{Riedler}. We also define in this section the MDP associated to our controlled PDMP and that we study later on.  
In Section \ref{RelaxControlSection} we use the results of \cite{Papageorgiou} to define relaxed controlled PDMPs and relaxed MDPs in infinite dimension. 
Section \ref{ResultSection} gathers the main results of the paper. We show that the optimal control problems of PDMPs and of MDPs are equivalent. We build up a general framework in which the MDP is contracting. The value function is then shown to be continuous and existence of optimal relaxed control strategies is proved. We finally give in this section, some convexity assumptions under which an ordinary optimal control strategy can be retrieved.

The final Section \ref{OptoSection} is devoted to showing that the previous theoretical results apply to the model of Optogenetics previously introduced. Several variants of the model are discussed, the scope of the theoretical results being much larger than the model of Definition \ref{stoInfiniteHH}.

\section{Theoretical framework for the control of infinite-dimensional PDMPs } \label{DefSection}
\markboth{Theoretical framework}{}
\subsection{The enlarged process and assumptions}\label{EnlargedSpace}

In the present section we define the infinite-dimensional controlled PDMPs that we consider in this paper in a way that enables us to formulate control problems in which the three characteristics of the PDMP depend on an additional variable that we call the control parameter. In particular we introduce the {\it enlarged process} which enable us to address optimization problems in the subsequent sections.

Let $(\Omega,\mathcal{F},(\mathcal{F}_t)_{t\geq 0},\mathbb{P})$ be a filtered probability space satisfying the usual conditions. We consider a Gelfand triple ($V \subset H \subset V^*$) such that $H$ is a separable Hilbert space and $V$ a separable, reflexive Banach space continuously and densely embedded in $H$. The pivot space $H$ is identified with its dual $H^*$, $V^*$ is the topological dual of $V$. $H$ is then continuously and densely embedded in $V^*$. We will denote by $||\cdot||_V$, $||\cdot||_H$, and  $||\cdot||_{V^*}$ the norms on $V$, $H$, and $V^*$, by $(\cdot,\cdot)$ the inner product in $H$ and by $\langle\cdot,\cdot \rangle$ the duality pairing of ($V,V^*$). Note that for $v\in V$ and $h\in H$, $\langle h, v \rangle = (h,v)$.

Let $D$ be a finite set, the state space of the discrete variable and $Z$ a compact Polish space, the control space. Let $T>0$ be the finite time horizon. Intuitively a controlled PDMP $(v_t,d_t)_{t\in[0,T]}$ should be constructed on $H\times D$ from the space of \textit{ordinary control rules} defined as
\begin{equation*}
A := \{ a : (0,T) \rightarrow U \text{ measurable}\},
\end{equation*}
where $U$, the \textit{action space}, is a closed subset of $Z$. Elements of $A$ are defined up to a set in $[0,T]$ of Lebesgue measure 0. The control rules introduced above are called \textit{ordinary} in contrast with the \textit{relaxed} ones that we will introduce and use in order to prove existence of optimal strategies.  When endowed with the coarsest $\sigma$-algebra such that 
\begin{equation*}
a \rightarrow \int_0^T e^{-t}w(t,a(t)) dt 
\end{equation*}
is measurable for all bounded and measurable functions $w:\mathbb{R}_+ \times U \rightarrow \mathbb{R}$, the set of control rules $A$ becomes a Borel space (see \cite[Lemma 1]{Yushkevich}). This will be crucial for the discrete-time control problem that we consider later. Conditionally to the continuous component $v_t$ and the control $a(t)$, the discrete component $d_t$ is a continuous-time Markov chain given by a jump rate function  $\lambda : H\times D\times U\rightarrow \mathbb{R}_+$ and a transition measure $\mathcal{Q} : H\times D\times U\rightarrow \mathcal{P}(D)$.

Between two consecutive jumps of the discrete component, the continuous component $v_t$ solves a controlled semilinear parabolic PDE
\begin{equation}\label{PDE}
\left\{
\begin{aligned}
\dot{v}_t  &= -L v_t + f_d(v_t,a(t)),\\
v_0 &= v, \quad v \in V.
\end{aligned}
\right.
\end{equation}
For $(v,d,a)\in H\times D\times A$ we will denote by $\phi^a(v,d)$ the flow of (\ref{PDE}). Let $T_n,n\in\mathbb{N}$ be the jump times of the PDMP. Their distribution is then given by 

\begin{equation}\label{jumpRate}
\mathbb{P}[d_{t+s} = d_t, 0\leq s \leq \Delta t | d_t] = \exp\left(-\int_0^{\Delta t} \lambda\Big(\phi^a_{t+s-T_n}(v_{T_n},d_{T_n}),d_t,a(t+s-T_n)\Big)\mathrm{d}s\right),
\end{equation}		
for $t\in[T_n;T_{n+1})$. When a jump occurs, the distribution of the post jump state is given by 

\begin{equation}\label{jumpMeasure}
\mathbb{P}[d_t = d | d_{t^-} \neq d_t] = \mathcal{Q}(\{d\}|d_t,v_t,a(t)).
\end{equation}

The triple ($\lambda, \mathcal{Q}, \phi$) fully describes the process and is referred to as the local characteristics of the PDMP.

\medskip

We will make the following assumptions on the local characteristics of the PDMP.

\begin{description}
\item[\namedlabel{hyp:lambda}{(H($\lambda$))}] For every $d\in D$, $\lambda_d: H\times Z \rightarrow \R_+$ is a function such that:
\begin{enumerate}
\item There exists $M_\lambda, \delta >0 $  such that: 
\[
\delta \leq \lambda_d(x,z) \leq M_\lambda,\quad \forall (x,z)\in H\times Z.
\]
\item $z\rightarrow \lambda_d(x,z)$ is continuous on $Z$, for all $x\in H$.
\item $x\rightarrow \lambda_d(x,z)$ is locally Lipschitz continuous, uniformly in $Z$, that is, for every compact set $K\subset H$, there exists $l_{\lambda}(K)>0$ such that
\[
|\lambda_d(x,z)-\lambda_d(y,z)| \leq l_{\lambda}(K) ||x-y||_H \quad \forall (x,y,z)\in K^2\times Z.
\]

\end{enumerate}
 
\item[\namedlabel{hyp:Q}{(H($\mathcal{Q}$))}] The function $\mathcal{Q} : H\times D\times Z\times \mathcal{B}(D) \rightarrow [0,1]$ is a transition probability such that:
$(x,z)\rightarrow \mathcal{Q}(\{p\} | x,d,z)$ is continuous for all $(d,p)\in D^2$ (weak continuity) and $\mathcal{Q}(\{d\}|x,d,z)=0$ for all $(x,z)\in H\times Z $.
\item[\namedlabel{hyp:L}{(H($L$))}] $L:V\rightarrow V^*$ is such that: 
\begin{enumerate}
\item $L$ is linear, monotone;
\item $||Lx||_{V^*}\leq c + c_1 ||x||_V$ with $c>0$ and $c_1\geq 0$;
\item $\langle Lx,x \rangle \geq c_2 ||x||_{V}^2,\quad c_2 > 0$;
\item $-L$ generates a strongly continuous semigroup $(S(t))_{t\geq 0}$ on $H$ such that $S(t):H\to H$ is compact for every $t>0$. We will denote by $M_S$ a bound, for the operator norm, of the semigroup on $[0,T]$.
\end{enumerate}
\item[\namedlabel{hyp:f}{(H($f$))}] For every $d\in D$, $f_d:H\times Z \rightarrow H$ is a function such that:
\begin{enumerate}

\item $x\rightarrow f_d(x,z)$ is Lipschitz continuous, uniformly in $Z$, that is,
\[
||f_d(x,z)-f_d(y,z)||_H \leq l_{f} ||x-y||_H \quad \forall (x,z)\in H\times Z, \quad l_f > 0.
\]
\item $(x,z) \rightarrow f_d(x,z)$ is continuous from $H\times Z$ to $H_w$, where $H_w$ denotes the space  $H$ endowed with the topology of weak convergence.
\end{enumerate}
\end{description}

Let us make some comments on the assumptions above. Assumption \ref{hyp:lambda}1. will ensure that the process is \textit{regular}, i.e. the number of jumps of $d_t$ is almost surely finite in every finite time interval. Assumption \ref{hyp:lambda}2. will enable us to construct relaxed trajectories. Assumptions \ref{hyp:lambda}3. and \ref{hyp:Q} will be necessary to obtain the existence of optimal relaxed controls for the associated MDP. Assumptions \ref{hyp:L}1.2.3. \ref{hyp:f} will ensure the existence and uniqueness of the solution of (\ref{PDE}). Note that all the results of this paper are unchanged if assumption \ref{hyp:f}1 is replaced by 

\begin{description}
\item[\namedlabel{hyp:fprim}{(H(f))'}] For every $d\in D$, $f_d:H\times Z \rightarrow H$ is a function such that:
\begin{enumerate}
\item $x\rightarrow -f_d(x,z)$ is continuous monotone, for all $z\in Z$.
\item $||f_d(x,z)||_H \leq b_1 + b_2||x||_H, \quad b_1\geq 0, b_2>0$, for all $z\in Z$.
\end{enumerate}
\end{description}

In particular, assumption \ref{hyp:f} implies \ref{hyp:fprim}2. and we will use the constants $b_1$ and $b_2$ further in this paper. Note that they can be chosen uniformly in $D$ since it is a finite set. To see this, note that $z\to f_d(0,z)$ is a weakly continuous on the compact space $Z$ and thus weakly bounded. It is then strongly bounded by the Uniform Boundedness Principle.

Finally, assumptions \ref{hyp:f}3. and \ref{hyp:L}4. will respectively ensure the existence of relaxed solutions of (\ref{PDE}) and the strong continuity of theses solutions with regards to the relaxed control. For that last matter, the compactness of $Z$ is also required. The following theorem is a reminder that the assumption on the semigroup does not make the problem trivial since it implies that $L$ is unbounded when $H$ is infinite-dimensional.

\begin{theorem}{(see \cite[Theorem 4.29]{Engel_Nagel_Semigroups})}\label{ImmediateCompactness}
\begin{enumerate}
\item For a strongly continuous semigroup $(T(t))_{t\geq 0}$ the following properties are equivalent
\begin{enumerate}
\item $(T(t))_{t\geq 0}$ is immediately compact.
\item $(T(t))_{t\geq 0}$ is immediately norm continuous, and its generator has compact resolvent.
\end{enumerate}
\item Let $X$ be a Banach space. A bounded operator $A\in\mathcal{L}(X)$ has compact resolvent if and only if $X$ is finite-dimensional. 
\end{enumerate}
\end{theorem}

\medskip

We define $\mathcal{U}_{ad}((0,T),U) := \{a \in L^1((0,T),Z) | a(t) \in U \text{ a.e.}\}\subset A$ the space of \textit{admissible rules}. Because of \ref{hyp:L} and \ref{hyp:f}, for all $a\in \mathcal{U}_{ad}((0,T),U)$, (\ref{PDE}) has a unique solution belonging to $L^2((0,T),V)\cap H^1((0,T),V^*)$ and moreover, the solution belongs to $C([0,T],H)$ (see \cite{Papageorgiou} for the construction of such a solution). We will make an extensive use of the mild formulation of the solution of (\ref{PDE}), given by

\begin{equation}\label{flow}
\phi^{a}_t(v,d) = S(t)v + \int_0^t S(t-s) f_d(\phi^{a}_s(v,d),a(s))\mathrm{d}s,
\end{equation}
with $\phi^{a}_0(v,d)= v$. One of the keys in the construction of a controlled PDMP in finite or infinite dimension is to ensure that $\phi^{a}$ enjoys the flow property $ \phi^{a}_{t+s}(v,d) = \phi^{a}_s(\phi^{a}_t(v,d),d)$ for all $(v,d,a)\in H\times D\times \mathcal{U}_{ad}((0,T),U)$ and $(t,s)\in\mathbb{R}_+$. It is the flow property that guarantees the Markov property for the process. Under the formulation (\ref{flow}), it is easy to see that the solution $\phi^{a}$ cannot feature the flow property for any reasonable set of admissible rules. In particular, the jump process $(d_t,t\geq 0)$ given by (\ref{jumpRate}) and (\ref{jumpMeasure}) is not Markovian. Moreover in control problems, and especially in Markovian control problems, we are generally looking for feedback controls which depend only on the current state variable so that at any time, the controller needs only to observe the current state to be able to take an action. Feedback controls would ensure the flow property. However they impose a huge restriction on the class of admissible controls. Indeed, feedback controls would be functions $u: H\times D\rightarrow U$ and for the solution of (\ref{PDE}) to be uniquely determined, the function $x\rightarrow f_d(x,u(x,d))$ needs to be Lipschitz continuous. It would automatically exclude discontinuous controls and therefore would not be adapted to control problems. To avoid this issue, Vermes introduced piecewise open-loop controls (see \cite{Vermes}): after a jump of the discrete component, the controller observes the location of the jump, say $(v,d)\in H\times D$ and chooses a control rule $a \in \mathcal{U}_{ad}((0,T),U)$ to be applied until the next jump. The time elapsed since the last jump must then be added to the state variable in order to see a control rule as a feedback control. While Vermes \cite{Vermes} and Davis \cite{DavisBook} only add the last post jump location we also want to keep track of the time of the last jump in order to define proper controls for the Markov Decision Processes that we introduce in the next section, and to eventually obtain optimal feedback policies. According to these remarks, we now enlarge the state space and define control strategies for the enlarged process. We introduce first several sets that will be useful later on.

\begin{definition}\label{troisensembles}
Let us define the following sets $D(T,2) := \{(t,s)\in[0,T]^2 \mid t+s \leq T\}$, \\ $\Xi := H\times D\times D(T,2)\times H$ and $\Upsilon :=H\times D\times [0,T]$.
\end{definition}

\begin{definition}\label{ControlStrat}{Control strategies. Enlarged controlled PDMP. Survival function.}

\noindent a) The set $\mathcal{A}$ of admissible control strategies is defined by 

\[
\mathcal{A} := \{\alpha: \Upsilon \rightarrow \mathcal{U}_{ad}([0,T];U) \text{ measurable} \}.
\]

\noindent b) On $\Xi$ we define the enlarged controlled PDMP $(X^{\alpha}_t)_{t\geq 0}=(v_t,d_t,\tau_t,h_t, \nu_t)_{t\geq 0}$ with strategy $\alpha\in\mathcal{A}$ as follows:
\begin{itemize}
\item $(v_t,d_t)_{t\geq 0}$ is the original PDMP,
\item $\tau_t$ is the time elapsed since the last jump at time $t$,
\item $h_t$ is the time of the last jump before time $t$,
\item $\nu_t$ is the post jump location right after the jump at time $h_t$. 
\end{itemize}

\noindent c) Let $z:=(v,d,h)\in\Upsilon$. For $a\in\mathcal{U}_{ad}([0,T];U)$ we will denote by $\chi_.^{a}(z)$ the solution of

\begin{equation*}
\frac{\mathrm{d}}{\mathrm{d}t} \chi_t^{a}(z) = - \chi_t^{a}(z) \lambda_{d}(\phi_t^{a}(z),a(t)), \qquad \chi_0^{a}(z) = 1, 
\end{equation*}

and its immediate extension  $\chi_.^{\alpha}(z)$ to $\mathcal{A}$ such that the process $(X_t^{\alpha})_{t\geq 0}$ starting at $(v,d,0,h,v)\in \Xi,$ admits $\chi_.^{\alpha}$ as survival function:
\[
\mathbb{P}[T_1>t] = \chi_t^{\alpha}(z).
\]
The notation $\phi_t^{a}(z)$ means here 

\[
\phi_t^{a}(z) := S(t)v + \int_0^t S(t-s)f_d(\phi_s^a(z),a(s))\mathrm{d}s.
\]
and $\phi_t^{\alpha}(z)$ means

\[
\phi_t^{\alpha}(z) := S(t)v + \int_0^t S(t-s)f_d(\phi_s^a(z),\alpha(z)(s))\mathrm{d}s.
\]

\end{definition}

\medskip

\begin{remark}\label{remControl}

\noindent i)Thanks to \cite[Lemma 3]{Yushkevich}, the set of admissible control strategies can be seen as a set of measurable feedback controls acting on $\Xi$ and with values in $U$. The formulation of Definition \ref{ControlStrat} is adequate to address the associated discrete-time control problem in Section \ref{MDP}.

\noindent ii) In view of Definition \ref{ControlStrat}, given $\alpha\in \mathcal{A}$, the deterministic dynamics of the process $(X^{\alpha}_t)_{t\geq 0} = (v_t,d_t,\tau_t,,h_t,\nu_t)_{t\geq 0}$ between two consecutive jumps obeys the initial value problem
\begin{equation}\label{PDEenlarged}
\left\{
\begin{aligned}
\dot{v}_t &= -L v_t + f_d(v_t,\alpha(v,d,s)(\tau_t)), \qquad v_s = v \in E,\\
\dot{d}_t & = 0,\qquad d_s = d\in D,\\
\dot{\tau}_t &= 1, \qquad \tau_s = 0,\\
\dot{h}_t & = 0,\qquad h_s = s \in [0,T],\\
\dot{\nu}_t &= 0, \qquad \nu_s = v_s = v,
\end{aligned}
\right.
\end{equation}
with $s$ the last time of jump.
The jump rate function and transition measure of the enlarged PDMP are straightforwardly given by the ones of the original process and will be denoted the same (see Appendix \ref{iteration} for their expression).

\noindent iii) If the relation $t=h_t+\tau_t$ indicates that the variable $h_t$ might be redundant, recall that we keep track of it on purpose. Indeed, the optimal control will appear as a function of the jump times so that keeping them as a variable will make the control feedback.

\noindent iv) Because of the special definition of the enlarged process, for every control strategy in $\mathcal{A}$, the initial point of the process $(X_t^{\alpha})_{t\geq 0}$ cannot be any point of the enlarged state space $\Xi$. More precisely we introduce in Definition \ref{InitialPoint} below the space of \textit{coherent initial points}.
\end{remark}

\begin{definition}\label{InitialPoint}{Space of coherent initial points.}

\noindent Take $\alpha\in\mathcal{A}$ and $x:=(v_0,d_0,0,h_0,v_0)\in \Xi$  and extend the notation $\phi_t^{\alpha}(x)$ of Definition \ref{ControlStrat} to $\Xi$ by 

\[
\phi_t^{\alpha}(x) := S(t)v_0 + \int_0^t S(t-s)f_{d_0}(\phi_s^{\alpha}(x),\alpha(v_0,d_0,h_0)(\tau_s))\mathrm{d}s
\]
The set $\Xi^{\alpha}\subset\Xi$ of coherent initial points is defined as follows

\begin{equation}\label{coherent}
\Xi^{\alpha} := \{(v,d,\tau,h,\nu)\in \Xi \mid v=\phi_{\tau}^{\alpha}(\nu,d,0,h,\nu)\}.
\end{equation}
\end{definition}

\noindent Then we have for all $x:=(v_0,d_0,\tau_0,h_0,\nu_0)\in \Xi^{\alpha}$,

\[
\phi_t^{\alpha}(x) := S(t)v_0 + \int_0^t S(t-s)f_{d_0}(\phi_s^{\alpha}(x),\alpha(\nu_0,d_0,h_0)(\tau_s))\mathrm{d}s
\]

\noindent Note that $(X^{\alpha}_t)$ can be constructed like any PDMP by a classical iteration that we recall in Appendix \ref{iteration} for the sake of completeness.

\medskip

\begin{proposition}\label{flowproperty} {The flow property.}

\noindent Take $\alpha\in\mathcal{A}$ and  $x:=(v_0,d_0,\tau_0,h_0,\nu_0)\in \Xi^{\alpha}$. Then $\phi_{t+s}^{\alpha}(x) = \phi_{t}^{\alpha}(\phi_{s}^{\alpha}(x),d_s,\tau_s,h_s,\nu_s)$ for all $(t,s)\in\mathbb{R}_+^2$. 
\end{proposition}

\medskip

\begin{notation}
\noindent  Let $\alpha\in\mathcal{A}$. For $z\in\Upsilon$, we will use the notation $\alpha_s(z):=\alpha(z)(s)$. Furthermore, we will sometimes denote by $\mathcal{Q}_{\alpha}(\cdot|v,d)$  instead of $\mathcal{Q}(\cdot|v,d,\alpha_{\tau}(\nu,d,h))$ for all $(v,d,\tau,h,\nu)\in \mathcal{A}\times \Xi^{\alpha}$. 
\end{notation}

\subsection{A probability space common to all strategies} \label{CommonSpaceSection}

Up to now thanks to Definition \ref{ControlStrat} we can formally associate the PDMP $(X^{\alpha}_t)_{t\in \mathbb{R}_+}$ to a given strategy $\alpha\in\mathcal{A}$. However, we need to show that there exists a filtered probabily space satisfying the usual conditions under which, for every control strategy $\alpha \in \mathcal{A}$, the controlled PDMP $(X_t^{\alpha})_{t\geq 0}$ is a homogeneous strong Markov process. This is what we do in the next theorem which provides an extension of \cite[Theorem 4]{Riedler} to controlled infinite-dimensional PDMPs and some estimates on the continuous component of the PDMP.

\begin{theorem}\label{THEOGENERATOR} Under assumptions \ref{hyp:lambda}, \ref{hyp:Q}, \ref{hyp:L} and \ref{hyp:f} (or \ref{hyp:fprim}) are satisfied.

\noindent a) There exists a filtered probability space satisfying the usual conditions such that for every control strategy $\alpha\in \mathcal{A}$ the process $(X_t^{\alpha})_{t\geq0}$ introduced in Definition \ref{ControlStrat} is a homogeneous strong Markov process on $\Xi$ with extended generator $\mathcal{G}^{\alpha}$ given in Appendix \ref{proofTheorem1}.

\noindent b) For every compact set $K\subset H$, there exists a deterministic constant $c_K > 0$ such that for all control strategy $\alpha\in\mathcal{A}$ and initial point $x:=(v,d,\tau,h,\nu)\in \Xi^{\alpha}$, with $v \in K$, the first component $v_t^{\alpha}$ of the control PDMP $(X_t^{\alpha})_{t\geq 0}$ starting at $x$ is such that 
\[
\sup_{t\in [0,T]} ||v_t^{\alpha}||_H \leq c_K.
\]
\end{theorem}

The proof of Theorem \ref{THEOGENERATOR} is given in Appendix \ref{proofTheorem1}. In the next section, we introduce the MDP that will allow us to prove the existence of optimal strategies.

\subsection{A Markov Decision Process (MDP)}\label{MDP}

Because of the particular definition of the state space $\Xi$, the state of the PDMP just after a jump is in fact fully determined by a point in $\Upsilon$. In Appendix \ref{proofTheorem1} we recall the one-to-one correspondence between the PDMP on $\Xi$ and the included pure jump process $(Z_n)_{n\in\mathbb{N}}$ with values in $\Upsilon$. This pure jump process allows to define a Markov Decision Process $(Z'_n)_{n\in\mathbb{N}}$ with values in $\Upsilon\cup\{\Delta_{\infty}\}$, where $\Delta_{\infty}$ is a cemetery state added to $\Upsilon$ to define a proper MDP. In order to lighten the notations, the dependence on a control strategy $\alpha\in\mathcal{A}$ of both jump processes is implicit. The stochastic kernel $\mathcal{Q}'$ of the MDP satisfies 
\begin{equation}
\mathcal{Q}'(B\times C\times E | z , a) =  \int_0^{T-h} \rho_t\mathrm{d}t,
\end{equation}

for any $z:=(v,d,h)\in \Upsilon$, Borel sets $B\subset H$, $C\subset D$, $E\subset [0,T]$, and $a \in \mathcal{U}_{ad}([0,T],U)$, where 
\begin{equation*}
\rho_t:=\lambda_{d}(\phi_{t}^{a}(z),a(t)) \chi_t^{a}(z)\mathbf{1}_E(h+t)\mathbf{1}_B(\phi_{t}^{a}(z))\mathcal{Q}(C|\phi_{t}^{a}(z),d,a(t)),
 \end{equation*}
 with $\phi_{t}^{a}(z)$ given by (\ref{flow}) and $\mathcal{Q}'(\{\Delta_{\infty}\}|z ,a) = \chi_{T-h}^{a}(z)$, and $\mathcal{Q}'(\{\Delta_{\infty}\}|\Delta_{\infty},a) = 1$. The conditional jumps of the MDP $(Z'_n)_{n\in\mathbb{N}}$ are then given by the kernel $\mathcal{Q}'(\cdot | z, \alpha(z))$ for $(z,\alpha)\in\Upsilon\times\mathcal{A}$.
Note that $Z'_n= Z_n$ as long as $T_n \leq T$, where $T_n$ is the last component of $Z_n$. Since we work with Borel state and control spaces, we will be able to apply techniques of \cite{BertsekasShreve} for discrete-time stochastic control problems, without being concerned by measurability matters. See \cite[Section 1.2]{BertsekasShreve} for an illuminating discussion on these measurability questions.

\section{Relaxed controls}\label{RelaxControlSection}

Relaxed controls are constructed by enlarging the set of ordinary ones, in order to convexify the original system, and in such a way that it is possible to approximate relaxed strategies by ordinary ones. The difficulty in doing so is twofold. First, the set of relaxed trajectories should not be much larger than the original one. Second, the topology considered on the set of relaxed controls should make it a compact set and, at the same time, make the flow of the associated PDE continuous. Compactness and continuity are two notions in conflict so being able to achieve such a construction is crucial. Intuitively a relaxed control strategy on the action space $U$ corresponds to randomizing the control action: at time $t$, instead of taking a predetermined action, the controller will take an action with some probability, making the control a transition probability. This has to be formalized mathematically.

\bigskip

\noindent {\bf Notation and reminder.} $Z$ is a compact Polish space, $C(Z)$ denotes the set of all real-valued continuous, necessarily bounded, functions on $Z$, endowed with the supremum norm. Because $Z$ is compact, by the Riesz Representation Theorem, the dual space $[C(Z)]^*$ of $C(Z)$ is identified with the space $M(Z)$ of Radon measures on $\mathcal{B}(Z)$, the Borel $\sigma$-field of $Z$. We will denote by $M_+^1(Z)$ the space of probability measures on $Z$. The action space $U$ is a closed subset of $Z$. We will use the notations $L^1(C(Z)):=L^1((0,T),C(Z))$ and $L^{\infty}(M(Z)):=L^{\infty}((0,T),M(Z))$.

\subsection{Relaxed controls for a PDE}\label{relaxedcontrolsInfDim} 

Let $\mathcal{B}([0,T])$ denote the Borel $\sigma$-field of $[0,T]$ and $Leb$ the Lebesgue measure.
A transition probability from $([0,T],\mathcal{B}([0,T]), Leb)$ into $(Z,\mathcal{B}(Z))$ is a function $\gamma : [0,T]\times \mathcal{B}(Z) \rightarrow [0,1] $ such that
\begin{equation*}
\left\{
\begin{aligned}
&t\rightarrow \gamma(t,C) \text{ is measurable for all }C\in \mathcal{B}(Z),\\
& \gamma(t,\cdot) \in M_+^1(Z) \text{ for all } t\in [0,T]. 
\end{aligned}
\right.
\end{equation*}

\noindent We will denote by $\mathcal{R}([0,T],Z)$ the set of all transition probability measures from \\ $([0,T],\mathcal{B}([0,T]), Leb)$ into $(Z,\mathcal{B}(Z))$.

\noindent Recall that we consider the PDE (\ref{PDE}): 

\begin{equation}\label{PDEadmissible}
\dot{v}_t  = L v_t + f_d(v_t,a(t)), \quad v_0 = v, \quad v \in V, \quad a\in \mathcal{U}_{ad}([0,T],U).
\end{equation}

\noindent The relaxed PDE is then of the form

\begin{equation}\label{relaxedPDE}
\dot{v}_t  = L v_t + \int_Z f_d(v_t,u)\gamma(t)(\mathrm{d}u),\quad v_0 = v, \quad v \in V, \quad \gamma\in \mathcal{R}([0,T],U),
\end{equation}

\noindent where $\mathcal{R}([0,T],U) := \{\gamma\in \mathcal{R}([0,T],Z) | \gamma(t)(U) = 1 \text{ a.e. in }[0,T]  \}$ is the set of transition probabilities from $([0,T],\mathcal{B}([0,T]), Leb)$ into $(Z,\mathcal{B}(Z))$ with support in $U$. The integral part of (\ref{relaxedPDE}) is to be understood in the sense of Bochner-Lebesgue as we show now. The topology we consider on $\mathcal{R}([0,T],U)$ follows from \cite{Balder} and because $Z$ is a compact metric space, it coincides with the usual topology of relaxed control theory of \cite{WargaBook}. It is the coarsest topology that makes continuous all mappings

\begin{equation*}
\gamma \rightarrow \int_0^T \int_Z f(t,z)\gamma(t)(\mathrm{d}z)\mathrm{d}t \in \R,	
\end{equation*}

\noindent for every Carathéodory integrand $f:[0,T]\times Z\rightarrow \R$, a Carathéodory integrand being such that 

\begin{equation*}
\left\{
\begin{aligned}
&t\rightarrow f(t,z) \text{ is measurable for all }z\in Z,\\
& z\rightarrow f(t,z) \text{ is continuous a.e., }\\
& |f(t,z)| \leq b(t) \text{ a.e., with }b\in L^1((0,T),\R). 
\end{aligned}
\right.
\end{equation*}

This topology is called the weak topology on $\mathcal{R}([0,T],Z)$ but we show now that it is in fact metrizable. 
Indeed, Carathéodory integrands $f$ on $[0,T]\times Z$ can be identified with the Lebesgue-Bochner space $L^1(C(Z))$ via the application $t\rightarrow f(t,\cdot) \in L^1(C(Z))$. Now, since $M(Z)$ is a separable ($Z$ is compact), dual space (dual of $C(Z)$), it enjoys the Radon-Nikodym property. Using \cite[Theorem 1 p. 98]{Diestel}, it follows that $[L^1(C(Z))]^* = L^{\infty}( M(Z))$. Hence, the weak topology on $\mathcal{R}([0,T],Z)$ can be identified with the $w^*$-topology in $(L^{\infty}(M(Z)),L^1(C(Z)))$, the latter being metrizable since $L^1(C(Z))$ is a separable space (see \cite[Theorem 1 p. 426]{DunfordSchwartz}). This crucial property allows to work with sequences when dealing with continuity matters with regards to relaxed controls. 

Finally, by Alaoglu's Theorem, $\mathcal{R}([0,T],U)$ is $w^*$-compact in $L^{\infty}(M(Z))$, and the set of original admissible controls $\mathcal{U}_{ad}([0,T],U)$ is dense in $\mathcal{R}([0,T],U)$ (see \cite[Corollary 3 p. 469]{Balder}).

For the same reasons why (\ref{PDEadmissible}) admits a unique solution, by setting $\bar{f}_d(v,\gamma):=\int_Zf_d(v,u)\gamma(\mathrm{d}u)$, it is straightforward to see that $(\ref{relaxedPDE})$ admits a unique solution.
The following theorem gathers of \cite[Theorems 3.2 and 4.1]{Papageorgiou} and will be of paramount importance in the sequel. 

\begin{theorem}\label{theoPapageorgiou}
If assumptions \ref{hyp:L} and \ref{hyp:f} (or \ref{hyp:fprim}) hold, then 

\noindent a) the space of relaxed trajectories (i.e. solutions of \ref{relaxedPDE}) is a convex, compact set of $C([0,T],H)$. It is the closure in $C([0,T],H)$ of the space of original trajectories (i.e. solutions of \ref{PDEadmissible}).

\noindent b) The mapping that maps a relaxed control to the solution of (\ref{relaxedPDE}) is continuous from $\mathcal{R}([0,T],U)$ into $C([0,T],H)$.
\end{theorem}

\subsection{Relaxed controls for infinite-dimensional PDMPs}\label{relaxedPDMP}

First of all, note that since the control acts on all three characteristics of the PDMP, convexity assumptions on the fields $f_d(v,U)$ would not necessarily ensure existence of optimal controls as it does for partial differential equations. Such assumptions should also be imposed on the rate function and the transition measure of the PDMP. For this reason, relaxed controls are even more important to prove existence of optimal controls for PDMP. For what has been done for PDE above, we are now able to define relaxed PDMPs. The next definition is the relaxed analogue of Definition \ref{ControlStrat}.

\begin{definition}\label{RelaxedControlStrat}{Relaxed control strategies, relaxed local characteristics.}

\noindent a) The set $\mathcal{A}^{\mathcal{R}}$ of relaxed admissible control strategies for the PDMP is defined by 

\[
\mathcal{A}^{\mathcal{R}} := \{\mu: \Upsilon \rightarrow \mathcal{R}([0,T];U) \text{ measurable} \}.
\]

\noindent Given a relaxed control strategy $\mu\in\mathcal{A}^{\mathcal{R}}$ and $z\in \Upsilon$, we will denote by $\mu^z:=\mu(z)\in  \mathcal{R}([0,T];U)$ and $\mu^z_t$ the corresponding probability measure on $(Z,\mathcal{B}(Z))$.

\noindent b) For $\gamma\in M_+^1(Z)$, $(v,d)\in H\times D$ and $C\in\mathcal{B}(D)$, we extend the jump rate function and transition measure as follows

\begin{equation}\label{relaxedCharact}
\left\{
\begin{aligned}
\lambda_d(v,\gamma) &:= \int_Z \lambda_d(v,u)\gamma(\mathrm{d}u), \\
\mathcal{Q}(C|v,d,\gamma) &:= \left(\lambda_d(v,\gamma)\right)^{-1} \int_Z \lambda_d(v,u)\mathcal{Q}(C|v,d,u)\gamma(\mathrm{d}u),
\end{aligned}
\right.
\end{equation}

\noindent the expression for the enlarged process being straightforward. This allows us to give the relaxed survival function of the PDMP and the relaxed mild formulation of the solution of (\ref{relaxedPDE})

\begin{equation}\label{relaxedSurvival}
\left\{
\begin{aligned}
\frac{\mathrm{d}}{\mathrm{d}t} \chi_t^{\mu}(z) &= - \chi_t^{\mu}(z)\lambda_d(\phi_t^{\mu}(z),\mu_t^z), \qquad \chi_0^{\mu}(z) = 1, \\
\phi_t^{\mu}(z) &=S(t)v + \int_0^t\int_Z  S(t-s)f_{d}(\phi_s^{\mu}(z),u)\mu_s^z(\mathrm{d}u)\mathrm{d}s, 
\end{aligned}
\right.
\end{equation}

\end{definition}

\noindent for $\mu\in\mathcal{A}^{\mathcal{R}}$ and $z:=(v,d,h)\in\Upsilon$. For $\gamma\in\mathcal{R}([0,T],U)$, we will also use the following notation 

\begin{equation*}
\left\{
\begin{aligned}
\chi_t^{\gamma}(z) &= \exp\left(-\int_0^t \lambda_d(\phi_s^{\gamma}(z),\gamma(t))\right), \\
\phi_t^{\gamma}(z) &=S(t)v + \int_0^t\int_Z  S(t-s)f_{d}(\phi_s^{\gamma}(z),u)\gamma(s)(\mathrm{d}u)\mathrm{d}s, 
\end{aligned}
\right.
\end{equation*}

\noindent The following proposition is a direct consequence of Theorem \ref{THEOGENERATOR}b).

\begin{proposition}	
For every compact set $K\subset H$, there exists a deterministic constant $c_K > 0$ such that for all control strategy $\mu\in\mathcal{A}^{\mathcal{R}}$ and initial point $x:=(v,d,\tau,h,\nu)\in \Xi^{\alpha}$, with $v \in K$, the first component $v_t^{\mu}$ of the control PDMP $(X_t^{\mu})_{t\geq 0}$ starting at $x$ is such that 
\[
\sup_{t\in [0,T]} ||v_t^{\mu}||_H \leq c_K.
\]
\end{proposition}

\noindent The relaxed transition measure is given in the next section through the relaxed stochastic kernel of the MDP associated to our relaxed PDMP.

\subsection{Relaxed associated MDP}

Let $z:=(v,d,h)\in\Upsilon$ and $\gamma\in\mathcal{R}([0,T],U)$. The relaxed stochastic kernel of the relaxed MDP satisfies

\begin{equation}
\mathcal{Q}'(B\times C\times E|z,\gamma)=\int_0^{T-h} \tilde{\rho}_t\mathrm{d}t,
\end{equation}

\noindent for Borel sets $B\subset H$, $C\subset D$, $E\subset [0,T]$, where 
\begin{align*}
\tilde{\rho}_t & :=\chi_t^{\gamma}(z) \mathbf{1}_E(h+t)\mathbf{1}_B(\phi_{t}^{\gamma}(z))\,\int_Z \lambda_{d}\Big(\phi_{t}^{\mu}(z),u\Big) \mathcal{Q}\Big(C|\phi_{t}^{\mu}(z),d,u\Big)\gamma(t)(\mathrm{d}u),\\
& = \chi_t^{\gamma}(z) \mathbf{1}_E(h+t)\mathbf{1}_B(\phi_{t}^{\gamma}(z))\lambda_{d}\Big(\phi_{t}^{\gamma}(z),\gamma(t)\Big)\mathcal{Q}\Big(C|\phi_{t}^{\gamma}(z),d,\gamma(t)\Big)
 \end{align*}

\noindent and $\mathcal{Q}'(\{\Delta_{\infty}\}|z ,\gamma) = \chi_{T-h}^{\gamma}(z)$, and $\mathcal{Q}'(\{\Delta_{\infty}\}|\Delta_{\infty},\gamma) = 1$, with, as before, the conditional jumps of the MDP $(Z'_n)_{n\in\mathbb{N}}$ given by the kernel $\mathcal{Q}'(\cdot | z, \mu(z))$ for $(z,\mu)\in\Upsilon\times\mathcal{A}^{\mathcal{R}}$.

\section{Main results}\label{ResultSection}

Here, we are interested in finding optimal controls for optimization problems involving infinite-dimensional PDMPs. For instance, we may want to track a targeted "signal" (as a solution of a given PDE, see Section \ref{OptoSection}). To do so, we are going to study the optimal control problem of the imbedded MDP defined in Section \ref{MDP}. This strategy has been for example used in \cite{Rieder} in the particular setting of a decoupled finite-dimensional PDMP, the rate function being constant.

\subsection{The optimal control problem}\label{costs}

Thanks to the preceding sections we can consider ordinary or relaxed costs for the PDMP $X^\alpha$  or the MDP and their corresponding value functions. For $z:=(v,d,h)\in\Upsilon$ and $\alpha\in\mathcal{A}$ we denote by $ \mathbb{E}_{z}^{\alpha}$ the conditional expectation given that $X_{h}^{\alpha}=(v,d,0,h,v)$ and by $X_s^{\alpha}(\phi)$ the first component of $X_s^{\alpha}$. Furthermore, we denote by $X_s^{\alpha}:=(v_s,d_s,\tau_s,h_s,\nu_s)$, then the shortened notation $\alpha(X_s^{\alpha})$ will refer to $\alpha_{\tau_s}(\nu_s,d_s,h_s)$. Theses notations are straightforwardly extended to $\mathcal{A}^{\mathcal{R}}$. We introduce a running cost $c : H\times Z\rightarrow \mathbb{R}_+$ and a terminal cost $g: H \rightarrow \mathbb{R}_+$ satisfying 

\begin{description}
\item[\namedlabel{hyp:cost}{(H($c$))}]
$(v,z)\rightarrow c(v,z)$ and $v\rightarrow g(v)$ are nonnegative norm quadratic functions, that is there exists $(a,b,c,d,e,f,g,h,i,j)\in \R^9$ such that for $v,z\in H\times Z$,
\begin{align*}
c(v,u) &= a||v||^2_H + b\bar{d}(0,u)^2 + c||v||_H\bar{d}(0,u) + d ||v||_H + e\bar{d}(0,u) + f,\\
 g(v) &= h||v||^2_H + i ||v||_H + j,
\end{align*}
with $\bar{d}(\cdot,\cdot)$ the distance on $Z$.
\end{description}

\begin{remark}
This assumption might seem a bit restrictive, but it falls within the framework of all the applications we have in mind. More importantly, it can be widely loosened if we slightly change the assumptions of Theorem \ref{TheoremExistenceOptimalControl}. In particular, all the following results, up to Lemma \ref{lemBoundingFunctionPDMP}, are true and proved for continuous functions $c : H\times Z\rightarrow \mathbb{R}_+$ and $g: H \rightarrow \mathbb{R}_+$. See Remark \ref{remarkCost} below.
\end{remark}

\begin{definition}{Ordinary value function for the PDMP $\, X^\alpha$.}

\noindent For $\alpha\in\mathcal{A} \, , $ we define the ordinary expected total cost function $V_{\alpha}:\Upsilon\rightarrow \mathbb{R}$ and the corresponding value function $V$ as follows:
\begin{equation}\label{ordinaryexpectedcost}
V_{\alpha}(z) := \mathbb{E}_{z}^{\alpha}\left[\int_{h}^T c(X_s^{\alpha}(\phi),\alpha(X_s^{\alpha}))\mathrm{d}s + g(X_T^{\alpha}(\phi))\right], \quad z:=(v,d,h)\in \Upsilon,
\end{equation}
\begin{equation}\label{ordinaryvaluefunction}
V(z) = \inf_{\alpha\in\mathcal{A}} V_{\alpha}(z), \quad z\in \Upsilon.
\end{equation}
\end{definition}

\medskip

\noindent Assumption \ref{hyp:cost} ensures that $V_{\alpha}$ and $V$ are properly defined. 

\begin{definition}{Relaxed value function for the PDMP $X^{\mu}$.}

\noindent For $\mu\in\mathcal{A}^{\mathcal{R}}$ we define the relaxed expected cost function $V_{\mu} : \Upsilon \rightarrow \mathbb{R}$ and the corresponding relaxed value function $\tilde{V}$ as follows:

\begin{equation}\label{relaxedexpectedcost}
V_{\mu}(z) := \mathbb{E}_{z}^{\mu}\left[\int_{h}^T \int_Z c(X_s^{\mu}(\phi),u)\mu(X_s^{\mu})(\mathrm{d}u)\mathrm{d}s + g(X_T^{\mu}(\phi))\right], \quad z:=(v,d,h)\in \Upsilon,
\end{equation}

\begin{equation}\label{relaxedvaluefunction}
\tilde{V}(z) = \inf_{\mu\in\mathcal{A}^{\mathcal{R}}} V_{\mu}(z), \quad z\in \Upsilon.
\end{equation}

\end{definition}

\bigskip

\noindent We can now state the main result of this section. 

\begin{theorem}\label{TheoremExistenceOptimalControl}
Under assumptions \ref{hyp:lambda}, \ref{hyp:Q}, \ref{hyp:L}, \ref{hyp:f} and \ref{hyp:cost}, the value function $\tilde{V}$ of the relaxed optimal control problem on the PDMP is continuous on $\Upsilon$ and there exists an optimal relaxed control strategy $\mu^*\in\mathcal{A}^{\mathcal{R}}$ such that
\[
\tilde{V}(z) = V_{\mu^*}(z),\qquad \forall z\in\Upsilon.
\]
	
\end{theorem}

\begin{remark}
All the subsequent results that lead to Theorem \ref{TheoremExistenceOptimalControl} would be easily transposable to the case of a lower semicontinuous cost function. We would then obtain a lower semicontinuous value function.	
\end{remark}

The next section is dedicated to proving Theorem \ref{TheoremExistenceOptimalControl} via the optimal control of the MDP introduced before. Let us briefly sum up what we are going to do. We first show that the optimal control problem of the PDMP is equivalent to the optimal control problem of the MDP and that an optimal control for the latter gives an optimal control strategy for the original PDMP. We will then build up a framework, based on so called bounding functions (see \cite{Rieder}), in which the value function of the MDP is the fixed point of a contracting operator. Finally, we show that under the assumptions of Theorem \ref{TheoremExistenceOptimalControl}, the relaxed PDMP $X^\mu$ belongs to this framework.

\subsection{Optimal control of the MDP}

Let us define the ordinary cost $c'$ on $\Upsilon\cup \{\Delta_{\infty}\}\times\mathcal{U}_{ad}([0,T];U)$ for the MDP defined in Section \ref{MDP}. For $z:=(v,d,h)\in \Upsilon$ and $a\in\mathcal{U}_{ad}([0,T];U)$,

\begin{equation}\label{ordinarycprime}
c'(z,a):=\int_0^{T-h} \chi_s^{a}(z)\; c(\phi_{s}^{a}(z), a(s)) \mathrm{d}s+\chi_{T-h}^{a}(z) g(\phi_{T-h}^{a}(z)),
\end{equation}
and $c'(\Delta_{\infty},a):=0$.

\noindent Assumption \ref{hyp:cost} allows $c'$ to be properly extended to $\mathcal{R}([0,T],U)$ by the formula
\begin{equation}\label{relaxedCost}
c'(z,\gamma) =\int_0^{T-h} \chi_s^{\gamma}(z) \int_Z  c(\phi_s^{\gamma}(z),u)\gamma(s)(\mathrm{d}u) \mathrm{d}s+  \chi_{T-h}^{\gamma}(z)g(\phi_{T-h}^{\gamma}(z)),
\end{equation}
and $c'(\Delta_{\infty},\gamma) = 0$ for $(z,\gamma)\in \Upsilon\times \mathcal{R}([0,T],U)$. We can now define the expected cost function and value function for the MDP.

\begin{definition}\label{CostMDP}{Cost and value functions for the MDP  $\, (Z'_n)$.}

\noindent For $\alpha \in \mathcal{A}\, $ (resp. $\mu\in \mathcal{A}^{\mathcal{R}}$), we define the total expected cost $J_{\alpha}$ (resp. $J_{\mu}$) and the value function $J$ (resp. $J'$) 

\begin{align*}
J_{\alpha}(z) & =  \mathbb{E}_{z}^{\alpha} \left[ \sum_{n=0}^{\infty}  c'(Z'_n,\alpha(Z'_n))\right], & J_{\mu}(z) & =  \mathbb{E}_{z}^{\mu} \left[ \sum_{n=0}^{\infty}   c'(Z'_n,\mu(Z'_n))\right]\, , \\
J(z) & =  \inf_{\alpha\in\mathcal{A}} J^{\alpha}(z), & J'(z) & =  \inf_{\mu\in\mathcal{A}^{\mathcal{R}}} J_{\mu}(z),
\end{align*}
for $z\in\Upsilon$ and with $\alpha(Z'_n)$ (resp. $\mu(Z'_n)$) being elements of $\mathcal{U}_{ad}([0,T],U)$ (resp. $\mathcal{R}([0,T],U)$).
\end{definition}
\noindent The finiteness of theses sums will by justified later by Lemma \ref{bBound}.

\subsubsection{The equivalence Theorem}

In the following theorem we prove that the relaxed expected cost function of the PDMP equals the one of the associated MDP. Thus, the value functions also coincide. For the finite-dimensional case we refer the reader to \cite{DavisBook} or \cite{Rieder} where the discrete component of the PDMP is a Poisson process and therefore the PDMP is entirely decoupled. The PDMPs that we consider are fully coupled.
 
\begin{theorem}\label{EquivalenceTheo}
The relaxed expected costs for the PDMP and the MDP coincide: $V_{\mu}(z)=J_{\mu}(z)$ for all $z \in \Upsilon$ and relaxed control $\mu \in\mathcal{A}^{\mathcal{R}}$. Thus, the value functions $\tilde{V}$ and $J'$ coincide on $ \Upsilon$.
\end{theorem}

\begin{remark}
Since we have $\mathcal{A}\subset \mathcal{A}^{\mathcal{R}}$, the value functions $V_{\alpha}(z)$ and $J_{\alpha}(z)$ also coincide for all $z \in \Upsilon$  and ordinary control strategy $\alpha \in\mathcal{A}$
\end{remark}

\begin{proof} Let $\mu\in\mathcal{A}^{\mathcal{R}}$ and $z=(v,d,h)\in \Upsilon$ and consider the PDMP $X^{\mu}$ starting at $(v,d,0,h,v)\in\Xi^{\mu}$. We drop the dependence in the control in the notation and denote by $(T_n)_{n\in\N}$ the jump times, and $Z_n:=(v_{T_n},d_{T_n},T_n)\in \Upsilon$ the point in $\Upsilon$ corresponding to $X_{T_n}^{\mu}$.  Let $H_n = (Z_0,\dots,Z_n)$, $T_n\leq T$. For a purpose of concision we will rewrite $\mu^n := \mu(Z_n) \in \mathcal{R}([0,T],U)$ for all $n\in\mathbb{N}$.
\begin{align*}
V_{\mu}(z) & =  \mathbb{E}_{z}^{\mu} \left[ \sum_{n=0}^{\infty} \int_{T\wedge T_n}^{T\wedge T_{n+1}} \int_Z c(X_s^{\mu}(\phi),u)\mu^n_{s-T_n}(\mathrm{d}u)\mathrm{d}s + \mathbf{1}_{\{T_n\leq T < T_{n+1}\}}g(X_T^{\mu}(\phi))\right]\\
 &= \sum_{n=0}^{\infty}\mathbb{E}_{z}^{\mu} \left[ \mathbb{E}_{z}^{\mu} \left[ \int_{T\wedge T_n}^{T\wedge T_{n+1}}\int_Z c(X_s^{\mu}(\phi),u)\mu^n_{s-T_n}(\mathrm{d}u)\mathrm{d}s + \mathbf{1}_{\{T_n\leq T < T_{n+1}\}}g(X_T^{\mu}(\phi))|H_n \right] \right],
\end{align*}
all quantities being non-negative.
We want now to examine the two terms that we call ${\cal I}_1$ and ${\cal I}_2$ separately. For $n\in \mathbb{N},$ we start with 
\begin{equation*}
{\cal I}_1:=\mathbb{E}_{z}^{\mu} \left[ \int_{T\wedge T_n}^{T\wedge T_{n+1}} \int_Z c(X_s^{\mu}(\phi),u)\mu^n_{s-T_n}(\mathrm{d}u)\mathrm{d}s |H_n \right]
\end{equation*}
that we split according to $T_n\leq T <T_n+1$ or $T_{n+1}\leq T\, $ (if $\, T\leq T_n$, the corresponding term vanishes). Then 
\begin{align*}
{\cal I}_1&=\mathbf{1}_{\{T_n\leq T\}} \mathbb{E}_{z}^{\mu} \left[\int_{T_n}^T \int_Z c(X_s^{\mu}(\phi),u)\mu^n_{s-T_n}(\mathrm{d}u)\mathbf{1}_{\{T_{n+1}>T\}}\mathrm{d}s|H_n \right]\\
 &\qquad + \mathbb{E}_{z}^{\mu} \left[ \mathbf{1}_{\{T_{n+1}\leq T\}}\int_{T_n}^{T_{n+1}} \int_Z c(X_s^{\mu}(\phi),u)\mu^n_{s-T_n}(\mathrm{d}u)\mathrm{d}s|H_n \right].
\end{align*}
By the strong Markov property and the flow property, the first term on the RHS is equal to  
\begin{align*} 
 &\mathbf{1}_{\{T_n\leq T\}} \mathbb{E}_{z}^{\mu}  \Bigg[\int_0^{T-T_n} \int_Z c(X_{T_n+s}^{\mu}(\phi),u)\mu^n_{s}(\mathrm{d}u) \mathbf{1}_{\{T_{n+1}-T_n>T-T_n\}}\mathrm{d}s|H_n  \Bigg]\\ 
 & =  \mathbf{1}_{\{T_n\leq T\}}\chi_{T-T_n}^{\mu}(Z_n) \int_0^{T-T_n} \int_Z c(\phi_s^{\mu}(Z_n),u)\mu^n_s(\mathrm{d}u)\mathrm{d}s.
\end{align*}

\noindent Using the same arguments, the second term on the RHS of ${\cal I}_1$ can be written as
\begin{equation*}
 \mathbf{1}_{\{T_n\leq T\}}  \int_0^{T-T_n} \int_Z \lambda_{d_n}(\phi_t^{\mu}(Z_n),u)\mu_t^n(\mathrm{d}u)\chi_t^{\mu}(Z_n)
\int_0^t \int_Z c(\phi_s^{\mu}(Z_n),u)\mu_t^n(\mathrm{d}u)\mathrm{d}s\mathrm{d}t,
\end{equation*}

\noindent An integration by parts yields 
\begin{equation*}
 {\cal I}_1=\mathbf{1}_{\{T_n\leq T\}}  \int_0^{T-T_n}  \chi_t^{\mu}(Z_n) \int_Z c(\phi_t^{\alpha}(Z_n),u)\mu_t^n(\mathrm{d}u)\mathrm{d}t.
\end{equation*}
Moreover 
\begin{equation*}
 \mathcal{I}_2:=\mathbb{E}_{z}^{\mu} \left[ \mathbf{1}_{\{T_n\leq T < T_{n+1}\}}g(X_T^{\mu})|H_n \right]  =  \mathbf{1}_{\{T_n\leq T\}}\chi_{T-T_n}^{\mu}(Z_n)g(\phi_{T-T_n}^{\mu}(Z_n))
\end{equation*}
By definition of the Markov chain $(Z'_n)_{n\in\mathbb{N}}$ and the function $c'$, we then obtain for the total expected cost of the PDMP, 
\begin{align*}
V_{\mu}(z) & = \sum_{n=0}^{\infty}\mathbb{E}_{z}^{\mu} \Bigg[    \mathbf{1}_{\{T_n\leq T\}}   \int_0^{T-T_n}  \chi_t^{\mu}(Z_n) \int_Z c(\phi_t^{\alpha}(Z_n),u)\mu_t^n(\mathrm{d}u)\mathrm{d}t\\
& \qquad \qquad  + \mathbf{1}_{\{T_n\leq T\}}  \, \, \chi_{T-T_n}^{\mu}(Z_n)g(\phi_{T-T_n}^{\mu}(Z_n))\Bigg]\\
& =  \mathbb{E}_{z}^{\mu} \left[ \sum_{n=0}^{\infty} c'(Z'_n,\mu(Z'_n))\right] = J_{\mu}(z).
\end{align*}
\end{proof}

\subsubsection{Existence of optimal controls for the MDP}

We now show existence of optimal relaxed controls under a contraction assumption. We use the notation $\mathcal{R}:=\mathcal{R}([0,T];U)$ in the sequel. Let us also recall some notations regarding the different control sets we consider.

\begin{itemize}
\item $u$ is an element of the control set $U$.
\item $a:[0,T]\rightarrow U$ is an element of the space of admissible control rules $\mathcal{U}_{ad}([0,T],U)$
\item $\alpha : \Upsilon\rightarrow \mathcal{U}_{ad}([0,T],U)$ is an element of the space of admissible strategies for the original PDMP.
\item $\gamma : [0,T] \rightarrow M_+^1(Z)$ is an element of the space of relaxed admissible control rules $\mathcal{R}$.
\item $\mu : \Upsilon \rightarrow \mathcal{R}$ is an element of the space of relaxed admissible strategies for the relaxed PDMP.

\end{itemize}

The classical way to address the discrete-time stochastic control problem that we introduced in Definition \ref{CostMDP} is to consider an additional control space that we will call the space of Markovian policies and denote by $\Pi$. Formally $\Pi := \left(\mathcal{A}^\mathcal{R}\right)^{\N}$ and a Markovian control policy for the MDP is a sequence of relaxed admissible strategies to be applied at each stage. The optimal control problem is to find $\pi:=(\mu_n)_{n\in\N}\in\Pi$ that minimizes

\[
 J_{\pi}(z) := \mathbb{E}_{z}^{\pi} \left[ \sum_{n=0}^{\infty} c'(Z'_n,\mu_n(Z'_n))\right].
\]

Now denote by $J^*(z)$ this infimum. We will in fact prove the existence of a stationary optimal control policy that will validate the equality 

\[
J^*(z) = J'(z).
\]

Let us now define some operators that will be useful for our study and state the first theorem of this section. Let $w:\Upsilon \rightarrow \R$ a continuous function, $(z,\gamma,\mu)\in \Upsilon\times \mathcal{R}\times\mathcal{A}^{\mathcal{R}}$ and define

\begin{align*}
Rw(z,\gamma)&:= c'(z,\gamma) + (\mathcal{Q}'w)(z,\gamma),\\
\mathcal{T}_{\mu}w(z)  &:= c'(z,\mu(z)) + (\mathcal{Q}'w)(z,\mu(z))=  Rw(z,\mu(z)),\\
(\mathcal{T}w)(z) & := \inf_{\gamma \in \mathcal{R}} \left\{c'(z,\gamma) + (\mathcal{Q}'w)(z,\gamma)\right\} = \inf_{\gamma \in \mathcal{R}} Rw(z,\gamma),
\end{align*}
where $(\mathcal{Q}'w)(z,\gamma):=\int_{\Upsilon} w(x)\mathcal{Q}'(\mathrm{d}x|z,\gamma)$ which admits also the expression

\begin{equation*}
\int_0^{T-h} \chi_t^{\gamma}(z)\,\int_Z \lambda_{d}\Big(\phi_{t}^{\gamma}(z),u\Big)\int_{D} w\Big(\phi_{t}^{\gamma}(z),r,h+t\Big) \mathcal{Q}\Big(\mathrm{d}r|\phi_{t}^{\gamma}(z),d,u\Big)\gamma(t)(\mathrm{d}u)\mathrm{d}t. 
\end{equation*}

\begin{theorem}\label{TheoContractingMDP}
Assume that there exists a subspace $\mathbb{C}$ of the space of continuous bounded functions from $\Upsilon$ to $\R$ such that the operator $\mathcal{T}: \mathbb{C}\rightarrow \mathbb{C}$ is contracting and the zero function belongs to $\mathbb{C}$. Assume furthermore that $\mathbb{C}$ is a Banach space. Then $J'$ is the unique fixed point of $\mathcal{T}$ and there exists an optimal control $\mu^*\in\mathcal{A}^{\mathcal{R}}$ such that
\[
J'(z) = J_{\mu^*}(z), \qquad \forall z\in \Upsilon.
\]

\end{theorem}
All the results needed to prove this Theorem can be found in \cite{BertsekasShreve}. We break down the proof into the two following elementary propositions, suited to our specific problem. 
Before that, recall that from \cite[Proposition 9.1 p.216]{BertsekasShreve}, $\Pi$ is the adequate control space to consider since history-dependent policies does not improve the value function.

Let us now consider the $n$-stages expected cost function and value function defined by

\begin{equation*}
J_{n\pi}(z) := \mathbb{E}_z^{\pi}\left[\sum_{i=0}^{n-1} c'\Big(Z'_i,\mu_i(Z'_i)\Big)\right] \qquad J_{n}(z) :=\inf_{\pi \in \Pi}  \mathbb{E}_z^{\pi}\left[\sum_{i=0}^{n-1} c'\Big(Z'_i,\mu_i(Z'_i)\Big)\right]
\end{equation*}
for $n\in\N$ and $\pi:=(\mu_n)_{n\in\N}\in\Pi$. We also set $J_{\infty} := \lim_{n\to \infty} J_n$.

\begin{proposition}

Let assumptions of Theorem \ref{TheoremExistenceOptimalControl} hold. Let $v,w: \Upsilon\to \R$ such that $v\leq w$ on $\Upsilon$, and let $\mu\in\mathcal{A}^{\mathcal{R}}$. Then $\mathcal{T}_{\mu}v \leq \mathcal{T}_{\mu}w$.
Moreover 
\[
J_n(z) = \inf_{\pi\in\Pi} (\mathcal{T}_{\mu_0}\mathcal{T}_{\mu_1}\dots\mathcal{T}_{\mu_{n-1}} 0)(z) = (\mathcal{T}^n0)(z),
\]
with $\pi:=(\mu_n)_{n\in\N}$ and $J_{\infty}$ is the unique fixed point of $\mathcal{T}$ in $\mathbb{C}$.
\end{proposition}

\begin{proof}
The first relation is straightforward since all quantities defining $\mathcal{Q}'$ are nonnegative. The equality $J_n = \inf_{\pi\in\Pi} \mathcal{T}_{\mu_0}\mathcal{T}_{\mu_1}\dots\mathcal{T}_{\mu_{n-1}} 0$ is also immediate since $\mathcal{T}_{\mu}$ just shifts the process of one stage (see also \cite[Lemma 8.1, p194]{BertsekasShreve}).

Let $I\in\mathbb{C}$, $\varepsilon > 0$ and $n\in\N$. For every $k\in \{1..n-1\}$, $\mathcal{T}^kI\in\mathbb{C}$ and so there exist $\mu_0,\mu_1,\dots,\mu_{n-1}\in\left(\mathcal{A}^{\mathcal{R}}\right)^n$ such that 

\begin{equation*}
\mathcal{T}_{\mu_{n-1}}I \leq \mathcal{T}I + \varepsilon, \quad \mathcal{T}_{\mu_{n-2}}\mathcal{T}I \leq \mathcal{T}\mathcal{T}I + \varepsilon, \quad \dots \quad ,
\mathcal{T}_{\mu_{0}}\mathcal{T}^{n-1}I \leq \mathcal{T}\mathcal{T}^{n-1}I + \varepsilon.
\end{equation*}
We then get

\begin{align*}
\mathcal{T}^nI  &\geq \mathcal{T}_{\mu_{0}}\mathcal{T}^{n-1}I-\varepsilon \geq  \mathcal{T}_{\mu_{0}} \mathcal{T}_{\mu_{1}}\mathcal{T}^{n-2}I -2\varepsilon\geq \quad \cdots \quad  \geq \mathcal{T}_{\mu_{0}}\mathcal{T}_{\mu_{1}}\dots\mathcal{T}_{\mu_{n-1}} I - n\varepsilon\\
& \geq \inf_{\pi\in\Pi} \mathcal{T}_{\mu_{0}}\mathcal{T}_{\mu_{1}}\dots\mathcal{T}_{\mu_{n-1}} I - n\varepsilon.
\end{align*}
Since this last inequality is true for any $\varepsilon>0$ we get 

\[
\mathcal{T}^nI \geq \inf_{\pi\in\Pi} \mathcal{T}_{\mu_{0}}\mathcal{T}_{\mu_{1}}\dots\mathcal{T}_{\mu_{n-1}} I,
\]
and by definition of $\mathcal{T}$, $\mathcal{T}I \leq \mathcal{T}_{\mu_{n-1}} I$. Using the first relation of the proposition we get 
\[
\mathcal{T}^{n}I \leq \mathcal{T}_{\mu_{0}}\mathcal{T}_{\mu_{1}}\dots\mathcal{T}_{\mu_{n-1}} I.
\]
Finally, $\mathcal{T}^nI = \inf_{\pi\in\Pi} \mathcal{T}_{\mu_{0}}\mathcal{T}_{\mu_{1}}\dots\mathcal{T}_{\mu_{n-1}} I$ for all $I\in \mathbb{C}$ and $n\in \N$. We deduce from the Banach fixed point theorem that $J_{\infty} = \lim_{n\to \infty} \mathcal{T}^n0$ belongs to $\mathbb{C}$ and is the only fixed point of $\mathcal{T}$.

\end{proof}	

\begin{proposition}
There exists $\mu^*\in\mathcal{A}^{\mathcal{R}}$ such that $J_{\infty} = J_{\mu^*} = J'$.	
\end{proposition}
\begin{proof}
By definition, for every $\pi\in\Pi$, $J_n \leq J_{n\pi}$, so that $J_{\infty} \leq J^*$. Now from the previous proposition, $J_{\infty} = \inf_{\gamma\in\mathcal{R}} LJ_{\infty}(\cdot,\gamma)$, $\mathcal{R}$ is a compact space and $LJ_{\infty}$ is a continuous function. We can thus find a measurable mapping $\mu^* : \Upsilon \to \mathcal{R}$ such that $J_{\infty} = \mathcal{T}_{\mu*}J_{\infty}$. $J_{\infty}\geq 0$ so from the first relation of the previous proposition, for all $n\in\N$, $J_{\infty} = \mathcal{T}_{\mu*}^nJ_{\infty} \geq \mathcal{T}_{\mu*}^n 0 $ and by taking the limit $J_{\infty} \geq J_{\mu^*}$. Since $J_{\mu^*} \geq J^*$ we get $J_{\infty} = J_{\mu^*} = J^*$.
We conclude the proof by remarking that $J^* \leq J' \leq J_{\mu^*}$.
\end{proof}

The next section is devoted to proving that the assumptions of Theorem \ref{TheoContractingMDP} are satisfied for the MDP. 

\subsubsection{Bounding functions and contracting MDP}

The concept of bounding function that we define below will ensure that the operator $\mathcal{T}$ is a contraction. The existence of the space $\mathbb{C}$ of Theorem \ref{TheoContractingMDP} will mostly result from Theorem \ref{theoPapageorgiou} and again from the concept of bounding function.
\begin{definition}\label{boundingFunction}{Bounding functions for a PDMP.}

Let $c$ (resp. $g$) be a running (resp. terminal) cost as in Section \ref{costs}. A measurable function $b:H\rightarrow \mathbb{R}_+$ is called a bounding function for the PDMP if there exist constants $c_c,c_g,c_{\phi}\in \mathbb{R}_+$ such that

\noindent (i) $c(v,u)\leq c_cb(v)$ for all $(v,u)\in H\times Z$,

\noindent (ii) $g(v) \leq c_gb(v)$ for all $v\in H$,

\noindent (iii) $b(\phi_t^{\gamma}(z))\leq c_{\phi}b(v)$ for all $(t,z,\gamma)\in [0,T]\times\Upsilon\times \mathcal{R}$, $z=(v,d,h)$.
\end{definition}

Given a bounding function for the PDMP we can construct one for the MDP with or without relaxed controls, as shown in the next lemma (cf. \cite[Definition 7.1.2 p.195]{BauerleRiederBook}).

\begin{lemma}\label{definitionBgamma} Let $b$ is a bounding function for the PDMP. We keep the notations of Definition \ref{boundingFunction}. Let $\zeta>0$. The function $B_{\zeta} : \Upsilon \longmapsto \mathbb{R}_+$ defined by $B_{\zeta}(z):=b(v)e^{\zeta(T-h)}$ for $z=(v,d,h)$ is an upper bounding function for the MDP. The two inequalities below are satisfied for all $(z,\gamma)\in \Upsilon\times \mathcal{R}$,
\begin{equation}\label{boundingc'}
c'(z,\gamma)\leq B_{\zeta}(z)c_{\phi}\left(\frac{c_c}{\delta}+c_g\right),
\end{equation} 
\begin{equation}\label{Bzeta}
\int_{\Upsilon} B_{\zeta}(y)\mathcal{Q}'(\mathrm{d}y|z,\gamma)\leq B_{\zeta}(z)\; c_{\phi}\;  \frac{M_{\lambda}}{(\zeta+\delta)}\;.
\end{equation}
\end{lemma}

\begin{proof} Take $(z,\gamma)\in \Upsilon\times \mathcal{R}$ , $z=(v,d,h)$. On the one hand from (\ref{relaxedCost}) and Definition \ref{boundingFunction} we obtain 
\begin{eqnarray*}
c'(z,\gamma) &\leq& \int_0^{T-h} e^{-\delta s}c_cc_{\phi}b(v)\mathrm{d}s + e^{-\delta(T-h)}c_gc_{\phi}b(v)\\ 
&\leq& B_{\zeta}(z)e^{-\zeta(T-h)}c_{\phi}\left(c_c\frac{1-e^{-\delta(T-h)}}{\delta} + e^{-\delta(T-h)}c_g\right),
\end{eqnarray*}
which immediately implies (\ref{boundingc'}). On the other hand 
\begin{align*}
\int_{\Upsilon} B_{\zeta}(y)\mathcal{Q}'(\mathrm{d}y|z,\gamma) & = \int_0^{T-h} \chi_s^{\gamma}(z)b(\phi_s^{\gamma}(z))e^{\zeta(T-h-s)}  \int_Z\lambda_d(\phi_s^{\gamma}(z),u)\mathcal{Q}(D|\phi_s^{\gamma}(z),u)\gamma_s(\mathrm{d}u)\mathrm{d}s\\
 & \leq  e^{\zeta(T-h)}b(v)c_{\phi}M_{\lambda} e^{-\zeta \tau}\int_0^{T-h} e^{-\delta s}e^{-\zeta s}\mathrm{d}s\\
 & = B_{\zeta}(z) c_{\phi}\frac{M_{\lambda}}{\zeta + \delta}\left(1-e^{-(\zeta+\delta)(T-h)}\right)
 \end{align*}
 which implies (\ref{Bzeta}).
 
\end{proof}

\bigskip

Let $b$ be a bounding function for the PDMP. Consider $\zeta^*$ such that $\,C:= c_{\phi}\frac{M_{\lambda}}{\zeta^* + \delta}\; < 1$. Denote by $B^*$ the associated bounding function for the MDP. We introduce the Banach space

\begin{equation}\label{bornepourB}
{\cal L}^*:=\{v : \Upsilon \rightarrow \mathbb{R} \text{ continuous } ; \, \, ||v||_* := \sup_{z\in\Upsilon} \frac{|v(z)|}{|B^*(z)|} < \infty \}\, .
\end{equation}


The following two lemmas give an estimate on the expected cost of the MDP that justifies manipulations of infinite sums.

\begin{lemma}\label{bBound} 
The inequality $\mathbb{E}_{z}^{\gamma}\left[B^*(Z'_k)\right] \leq C^k B^*(z)$ holds for any $(z,\gamma,k)\in \Upsilon\times \mathcal{R}\times \mathbb{N}$.
\end{lemma}

\begin{proof}
We proceed by induction on $k$. Let $z\in\Upsilon$. The desired inequality holds for $k=0$ since $\mathbb{E}_{z}^{\gamma}\left[B^*(Z'_0)\right]=B^*(z)$. Suppose now that it holds for $k\in \mathbb{N}$. Then   
\begin{eqnarray*}
\mathbb{E}_{z}^{\gamma}\left[B^*(Z'_{k+1})\right] &=& \mathbb{E}_{z}^{\gamma}\left[ \mathbb{E}_{z}^{\gamma}\left[ B^*(Z'_{k+1})|Z'_k\right] \right] \\
 & =&  \mathbb{E}_{z}^{\gamma}\left[ \int_{\Upsilon} B^*(y) \mathcal{Q}'(\mathrm{d}y|Z'_k,\gamma )\right]\\
 & =& \mathbb{E}_{z}^{\gamma}\left[ B^*(Z'_k)\frac{\int_{\Upsilon} B^*(y) \mathcal{Q}'(\mathrm{d}y|Z'_k,\gamma )}{B^*(Z'_k)}\right].
\end{eqnarray*} 
Using (\ref{Bzeta}) and the definition of $C$, we conclude that $\mathbb{E}_{z}^{\gamma}\left[B^*(Z'_{k+1})\right]\leq C \mathbb{E}_{z}^{\gamma}\left[ B^*(Z'_k)\right]$ and by the assumption on $k$ $\mathbb{E}_{z}^{\gamma}\left[B^*(Z'_{k+1})\right]\leq C^{k+1}B^*(z)$.
\end{proof}

\begin{lemma}
There exists $\kappa>0$ such that for any $(z,\mu)\in \Upsilon \times \mathcal{A}^\mathcal{R}$,
\begin{equation*}
\mathbb{E}_{z}^{\mu}\left[ \sum_{k=n}^{\infty} c'(Z'_k,\mu(Z'_k)) \right] \leq \kappa\, \frac{C^n}{1-C}\, B^*(z).
\end{equation*}
\end{lemma}

\begin{proof}
The results follows from Lemma \ref{bBound} and from the fact that $$c'(Z'_k,\mu(Z'_k)) \leq B^*(Z_k)c_{\phi}\left(\frac{c_c}{\delta}+c_g\right)$$ for any $k\in \mathbb{N}$.
\end{proof}
We now state the result on the operator $\mathcal{T}$.

\begin{lemma}\label{Tcontracting}
$\mathcal {T}$ is a contraction on ${\cal L}^*$: for any $(v,w)\in {{\cal L}^*}\times {{\cal L}^*}$,
\begin{equation*}
||\mathcal{T}v-\mathcal{T}w||_{B^*} \, \leq \, C \, ||v-w||_{B^*},
\end{equation*}
where $C=c_{\phi}\frac{M_{\lambda}}{\zeta^* + \delta}$.
\end{lemma}

\begin{proof}
We prove here the contraction property. The fact $\mathcal{T}: \mathcal{L}^*\to\mathcal{L}^*$ is less straightforward and is addressed in the next section.
Let $z:=(v,d,h)\in\Upsilon$. Let us recall that for functions $f,g:\mathcal{R}\rightarrow \mathbb{R}$
\begin{equation*}
\sup_{\gamma\in \mathcal{R}} f(\gamma) - \sup_{\gamma\in \mathcal{R}} g(\gamma) \leq  \sup_{\gamma\in \mathcal{R}} \left( f(\gamma) -g(\gamma)\right).
\end{equation*}
Moreover since $\inf_{\gamma\in \mathcal{R}} f(\gamma) - \inf_{\gamma\in \mathcal{R}} g(\gamma) = \sup_{\gamma\in \mathcal{R}} (-g(\gamma)) - \sup_{\gamma\in \mathcal{R}} (-f(\gamma))$, we have 
\begin{equation*}
\mathcal{T}v\, (z)-\mathcal{T}w\, (z)  \leq \sup_{\gamma\in \mathcal{R}} \int_0^{T-h}  \chi_s^{\gamma}(z)\int_Z\lambda_{d}(\phi_s^{\gamma}(z),u){\cal I}(u,s)\, \gamma(s)(\mathrm{d}u)\mathrm{d}s,
\end{equation*}
where
\begin{equation*}
 {\cal I}(u,s):=\int_{D} \Big(v(\phi_s^{\gamma}(z),r,h+s)-w(\phi_s^{\gamma}(z),r,h+s)\Big)
\mathcal{Q}(\mathrm{d}r|\phi_s^{\gamma}(z),d,u),
\end{equation*}
so that  
\begin{equation*}
||\mathcal{T}v-\mathcal{T}w||_{B^*}  \leq \sup_{(z,\gamma)\in \Upsilon\times\mathcal{R}} \int_0^{T-h} \chi_s^{\gamma}(z)\int_Z\lambda_{d}(\phi_s^{\gamma}(z),u){\cal J}(s,u)\gamma(s)(\mathrm{d}u)\mathrm{d}s
\end{equation*}
where 
\begin{equation*}
{\cal J}(s,u):=\int_{D} \frac{B^* (\phi_s^{\gamma}(z),r,h+s)}{B^* (z)}||v-w||_{B^*} \mathcal{Q}(\mathrm{d}r|\phi_s^{\gamma}(z),d,u)
\end{equation*}
We then conclude that 
\begin{align*}
||\mathcal{T}v-\mathcal{T}w||_{B^*} & \leq  \sup_{(z,\gamma)\in \Upsilon\times\mathcal{R}} \int_0^{T-h} e^{-\delta s}M_{\lambda}c_{\phi}e^{-\zeta^*s}\mathrm{d}s\, ||v-w||_{B^*}\\
& \leq M_{\lambda}c_{\phi}\, ||v-w||_{B^*} \int_0^{T-h} e^{-(\delta+\zeta^*) s}\mathrm{d}s\\
& \leq  C ||v-w||_{B^*}.
\end{align*}
\end{proof}

\subsubsection{Continuity properties}\label{RelaxTheoSection}

Here we prove that the trajectories of the relaxed PDMP are continuous w.r.t. the control and that the operator $R$ transforms continuous functions in continuous functions.

\begin{lemma}\label{ContinuityTheo} Assume that \ref{hyp:L} and \ref{hyp:f} are satisfied.
Then the mapping
\begin{equation*}
\phi : (z,\gamma)\in\Upsilon\times \mathcal{R} \rightarrow \phi_\cdot^{\gamma}(z) =S(0)v + \int_0^{\cdot}\int_Z S(\cdot-s) f_d(\phi_s^{\gamma}(z),u)\gamma(s)(\mathrm{d}u)\mathrm{d}s
\end{equation*}
is continuous from $\Upsilon\times \mathcal{R}$ in $C([0,T];H)$.
\end{lemma}

\begin{proof}
This proof is based on the result of Theorem \ref{theoPapageorgiou}. Here we add the joined continuity on $\Upsilon\times \mathcal{R}$ whereas the continuity is just on $\mathcal{R}$ in \cite{Papageorgiou}.
Let $t\in[0,T]$ and let $(z,\gamma) \in \Upsilon\times \mathcal{R}$. Assume that $(z_n,\gamma_n)\rightarrow (z,\gamma)$. Since $D$ is a finite set, we take the discrete topology on it and if we denote by $z^n=(v_n,d_n,h_n)$ and $z=(v,d,h)$, we have the equality $d_n=d$ for $n$ large enough. So for $n$ large enough we have
\begin{align*}
 \phi_t^{\gamma_n}(z_n) - \phi_t^{\gamma}(z) & = S(t)v_n-S(t)v + \int_0^t\int_Z S(t-s)f_d(\phi_t^{\gamma_n}(z_n),u)\gamma_n(s)(\mathrm{d}u)\mathrm{d}s \\
 & \quad - \int_0^t\int_Z S(t-s)f_d(\phi_t^{\gamma}(z),u)\gamma(s)(\mathrm{}du)\mathrm{d}s\\
 & = S(t)v_n-S(t)v \\
 & \quad + \int_0^t \int_Z S(t-s)[f_d(\phi_t^{\gamma_n}(z_n),u)\gamma_n(s)(\mathrm{d}u)-f_d(\phi_t^{\gamma}(z),u)\gamma_n(s)(\mathrm{d}u)]\mathrm{d}s\\
 & \quad + \int_0^t \int_Z S(t-s)[f_d(\phi_t^{\gamma}(z),u)\gamma_n(s)(\mathrm{d}u)-f_d(\phi_t^{\gamma}(z),u)\gamma(s)(\mathrm{d}u)]\mathrm{d}s.
 \end{align*}
From\ref{hyp:f}1. we get
 \begin{equation*}
 ||\phi_t^{\gamma_n}(z_n) - \phi_t^{\gamma}(z)||_H  \leq M_S||v_n-v||_H + M_Sl_f \int_0^t||\phi_s^{\gamma_n}(z_n) - \phi_s^{\gamma}(z)||_H\mathrm{d}s + ||\ell_n(t)||_H 
 \end{equation*}
 where $\ell_n(t) := \int_0^t \int_Z S(t-s)[f_d(\phi_t^{\gamma}(z),u)\gamma_n(s)(\mathrm{d}u)-f_d(\phi_t^{\gamma}(z),u)\gamma(s)(\mathrm{d}u)]\mathrm{d}s.$ By the Gronwall lemma we obtain a constant $C>0$ such that
 \[
  ||\phi_t^{\gamma_n}(z_n) - \phi_t^{\gamma}(z)||_H  \leq C(||v_n-v||_H+||\ell_n(t)||_H).
 \]
Since $\lim_{n\rightarrow +\infty}||v_n-v||_H=0$, the proof is complete if we show that the sequence of functions $(||\ell_n||_H)$ uniformly converges to 0. 

Let us denote by $x_n(t):= 	\int_0^t \int_Z (h, S(t-s)f_d(\phi_t^{\gamma}(z),u)))_H\gamma_n(s)(\mathrm{d}u)\mathrm{d}s $. Using the same argument as the proof of \cite[Theorem 3.1]{Papageorgiou}, there is no difficulty in proving that $(x_n)_{n\in\N}$ is compact in $C([0,T],H)$ so that, passing to a subsequence if necessary, we may assume that $x_n\to x$ in $C([0,T],H)$. Now let $h\in H$. 

\begin{align*}
(h,\ell_n(t))_H &= 	\int_0^t \int_Z (h, S(t-s)f_d(\phi_t^{\gamma}(z),u)))_H\gamma_n(s)(\mathrm{d}u)\mathrm{d}s \\
& \quad - \int_0^t \int_Z (h, S(t-s)f_d(\phi_t^{\gamma}(z),u)))_H\gamma(s)(\mathrm{d}u)\mathrm{d}s \xrightarrow[n\rightarrow \infty ]{} 0,
\end{align*}
since $(t,u) \to ( h, S(t-s)f_d(\phi_t^{\gamma}(z),u)))_H \in L^1(C(Z))$ and $\gamma_n\to \gamma$ weakly* in $L^{\infty}(M(Z)) = [L^1(C(Z))]^*$. Thus, $x(t) =  \int_0^t \int_Z S(t-s)f_d(\phi_t^{\gamma}(z),u)\gamma(s)(\mathrm{d}u)\mathrm{d}s$ and $\ell_n(t) = x_n(t)- x(t)$  for all $t\in[0,T]$, proving the uniform convergence of $||\ell_n||_H$ on $[0,T]$.

\end{proof}

The next lemma establishes the continuity property of the operator $R$.

\medskip

\begin{lemma}\label{CONTINUITYLEMMA}
Suppose that assumptions \ref{hyp:L}, \ref{hyp:f}, \ref{hyp:lambda}, \ref{hyp:Q}, \ref{hyp:cost} are satisfied.
Let $b$ be a continuous bounding function for the PDMP. Let $w:\Upsilon\times U \rightarrow \mathbb{R}$ be continuous with $|w(z,u)|\leq c_wB^*(z)$ for some $c_w\geq0$. Then

\[
(z,\gamma) \rightarrow \int_0^{T-h} \chi_s^{\gamma}(z)\left( \int_Z w(\phi_s^{\gamma}(z),d,h+s,u)\gamma(s)(\mathrm{d}u)\right)\mathrm{d}s
\]
is continuous on $\Upsilon\times \mathcal{R}$, with $z:=(v,d,h)$.
Quite straightforwardly, 

\[
(z,\gamma) \rightarrow Rw(z,\gamma) = c'(z,\gamma) + Q'w\, (z,\gamma)
\]
is continuous on $\Upsilon\times \mathcal{R}$.
\end{lemma}

\begin{proof}
See Appendix \ref{proofContinuity}.
\end{proof}

It now remains to show that there exists a bounding function for the PDMP. This is the result of the next lemma.

\begin{lemma}\label{lemBoundingFunctionPDMP}
Suppose assumptions \ref{hyp:L}, \ref{hyp:f} and \ref{hyp:cost} are satisfied. Now define $\tilde{c}$ and $\tilde{g}$ from $c$ and $g$ by taking the absolute value of the coefficients of these quadratic functions. Let $M_2>0$. Define $M_3 := (M_2+b_1T)M_Se^{M_Sb_2T}$ and $b:H\rightarrow \R_+$ by 

\begin{equation}
b(v) := \left\{
\begin{aligned}
&\max_{||x||_H\leq M_3}\max_{u\in U}\tilde{c}(x,u) + \max_{||x||_H\leq M_3} \tilde{g}(x) ,&\text{if  } ||v||_H\leq M_3,\\
& \max_{u\in U}\tilde{c}(v,u) + \tilde{g}(v),& \text{if  } ||v||_H > M_3,
\end{aligned}
\right.
\end{equation}
is a continuous bounding function for the PDMP.
\end{lemma}

\begin{proof}
For all $(v,u)\in H\times U$, $c(v,u)\leq b(v)$ and $g(v)\leq b(v)$. Now let $(t,z,\gamma)\in [0,T]\times \Upsilon\times \mathcal{R}$, $z=(v,d,h)$.
\begin{itemize}

\item If $||\phi_t^{\gamma}(z)||_H \leq M_3$, $b(\phi_t^{\gamma}(z)) = b(M_3)$. If $||v||_H \leq M_3$ then $b(v) = b(M_3) = b(\phi_t^{\gamma}(z))$. Otherwise, $||v||_H > M_3$ and $b(v) > b(M_3) = b(\phi_t^{\gamma}(z))$.

\item If $||\phi_t^{\gamma}(z)||_H > M_3$ then $||v||_H > M_2$ and $||\phi_t^{\gamma}(z)||_H \leq ||v||_HM_3/M_2$ (See \ref{phiBound} in Appendix \ref{proofTheorem1}). So,

\begin{equation*}
b(\phi_t^{\gamma}(z)))  = \max_{u\in U}\tilde{c}(\phi_t^{\gamma}(z),u) + \tilde{g}(\phi_t^{\gamma}(z)) \leq b\left(\frac{M_3}{M_2}v\right) \leq \frac{M_3^2}{M_2^2}b(v),
\end{equation*}

since $M_3/M_2 > 1$.

\end{itemize}

\end{proof}

\begin{remark}\label{remarkCost}
Lemma \ref{lemBoundingFunctionPDMP} ensures the existence of a bounding function for the PDMP. To broaden the class of cost functions considered, we could just assume the existence of a bounding for the PDMP in Theorem \ref{TheoremExistenceOptimalControl} and then, the assumption on $c$ and $g$ should just be the continuity.
\end{remark}

\subsection{Existence of an optimal ordinary strategy}

Ordinary strategies are of crucial importance because they are the ones that the controller can implement in practice.
Here we give convexity assumptions that ensure the existence of an ordinary optimal control strategy for the PDMP.

\begin{enumerate}[label=(A)] 
\item \label{HypOrdinaryControlExists}
\begin{enumerate}
\item  For all $d\in D$, the function $f_d : (y,u)\in H\times U \rightarrow E$ is linear in the control variable $u$.\label{HypLinearFlow}
\item For all $d\in D$, the functions $\lambda_d : (y,u) \in H\times U \rightarrow \mathbb{R}_+$ and $\lambda_d\mathcal{Q} : (y,u)\in H\times U \rightarrow \lambda_d(y,u)\mathcal{Q}(\cdot|y,d,u)$ are respectively concave and convexe in the control variable $u$. \label{HypConvexJumpRate}
\item The cost function $c:(y,u)\in E\times U \rightarrow \mathbb{R}_+ $ is convex in the control variable $u$.\label{HypConvexCost}
\end{enumerate}
\end{enumerate}

\begin{theorem}\label{TheoExistenceOrdinaryOptimalControl}
Suppose that assumptions \ref{hyp:L}, \ref{hyp:f}, \ref{hyp:lambda}, \ref{hyp:Q}, \ref{hyp:cost} and \ref{HypOrdinaryControlExists} are satisfied.  If we consider $\mu^*\in \mathcal{A}^\mathcal{R}$ an optimal relaxed strategy for the PDMP, then the ordinary strategy $\bar{\mu}_t := \int_Z u\mu^*_t(du)\in\mathcal{A}$  is optimal, i.e. $V_{\bar{\mu}}(z) = \tilde{V}_{\mu^*}(z) = V(z), \quad \forall z\in \Upsilon.$
\end{theorem}

\begin{proof}
This result is based on the fact that for all $(z,\gamma)\in \Upsilon\times \mathcal{R}, (Lw)(z,\gamma) \geq (Lw)(z,\bar{\gamma})$, with $\bar{\gamma} = \int_Z u\gamma(\mathrm{d}u).$ Indeed, the fact that the function $f_d$ is linear in the control variable implies that for all $(t,z,\gamma)\in [0,T]\times\Upsilon\times \mathcal{R},\phi_t^{\gamma}(z) =\phi_t^{\bar{\gamma}}(z)$. The convexity assumptions \ref{HypOrdinaryControlExists} give the following inequalities
\begin{align*}
 \int_Z \lambda_d(\phi_s^{\gamma}(z),u)\gamma(s)(\mathrm{d}u) &\leq \lambda_d(\phi_s^{\bar{\gamma}}(z),\bar{\gamma}(s)),\\
  \int_Z \lambda_d(\phi_s^{\gamma}(z),u)\mathcal{Q}(E|\phi_s^{\gamma}(z),d,u)\gamma(s)(\mathrm{d}u) & \geq  \lambda_d(\phi_s^{\bar{\gamma}}(z),\bar{\gamma}(s))\mathcal{Q}(E|\phi_s^{\bar{\gamma}}(z),d,\bar{\gamma}(s)),\\
 \int_Z c(\phi_s^{\gamma}(z),u)\gamma_s(\mathrm{d}u) &\geq c(\phi_s^{\bar{\gamma}}(z),\bar{\gamma_s}),
\end{align*}

for all $(s,z,\gamma,E)\in [0,T]\times\Upsilon\times \mathcal{R}\times \mathcal{B}(D)$, so that in particular $\chi_t^{\gamma}(z)\geq  \chi_t^{\bar{\gamma}}(z)$. We can now denote for all $(z,\gamma)\in \Upsilon\times \mathcal{R}$ and $w: \Upsilon\rightarrow \R_+$,

\begin{align*}
 (Lw)(z,\gamma)  & = \int_0^{T-h} \chi_s^{\gamma}(z) \int_Z c(\phi_s^{\gamma}(z),u)\gamma(s)(\mathrm{d}u) \mathrm{d}s +  \chi_{T-h}^{\gamma}(z)g(\phi_{T-h}^{\gamma}(z))\\
 & +  \int_0^{T-h}  \chi_s^{\gamma}(z) \int_Z \lambda_d(\phi_s^{\gamma}(z),u) \int_{D}w(\phi_s^{\gamma}(z),r,h+s)\mathcal{Q}(\mathrm{d}r|\phi_s^{\gamma}(z),d,u)\gamma(s)(\mathrm{d}u)\mathrm{d}s\\
 & \geq \int_0^{T-h}\chi_s^{\bar{\gamma}}(z)  c(\phi_s^{\bar{\gamma}}(z),\bar{\gamma}(s)) \mathrm{d}s +  \chi_{T-h}^{\bar{\gamma}}(z)g(\phi_{T-h}^{\bar{\gamma}}(z))\\
 & +  \int_0^{T-h}  \chi_s^{\bar{\gamma}}(z) \int_Z \lambda_d(\phi_s^{\bar{\gamma}}(z),u) \int_{D}w(\phi_s^{\bar{\gamma}}(z),r,h+s)\mathcal{Q}(\mathrm{d}r|\phi_s^{\bar{\gamma}}(z),d,u)\gamma(s)(\mathrm{d}u)\mathrm{d}s.\\
\end{align*}

\noindent Furthermore,

\begin{align*}
 \int_Z \lambda_d(\phi_s^{\bar{\gamma}}(z),u)& \int_Dw(\phi_s^{\bar{\gamma}}(z),r,h+s)\mathcal{Q}(\mathrm{d}r|\phi_s^{\bar{\gamma}}(z),d,u)\gamma(s)(\mathrm{d}u)\geq\\ 
 &\lambda_d(\phi_s^{\bar{\gamma}}(z),\bar{\gamma}(s)) \int_D w(\phi_s^{\bar{\gamma}}(z),r,h+s)\mathcal{Q}(\mathrm{d}r|\phi_s^{\bar{\gamma}}(z),d,\bar{\gamma}(s)),
\end{align*} so that

\begin{align*}
 (Lw)(z,\gamma) &\geq \int_0^{T-h}\chi_s^{\bar{\gamma}}(z)  c(\phi_s^{\bar{\gamma}}(z),\bar{\gamma}(s)) \mathrm{d}s +  \chi_{T-h}^{\bar{\gamma}}(z)g(\phi_{T-h}^{\bar{\gamma}}(z))\\
 & +  \int_0^{T-h}  \chi_s^{\bar{\gamma}}(z) \lambda_d(\phi_s^{\bar{\gamma}}(z),\bar{\gamma}(s)) \int_D w(\phi_s^{\bar{\gamma}}(z),r,h+s)\mathcal{Q}(\mathrm{d}r|\phi_s^{\bar{\gamma}}(z),d,\bar{\gamma}(s))\\
 & =  (Lw)(z,\bar{\gamma}).
\end{align*}
\end{proof}

\subsection{An elementary example}\label{elementaryExample}
Here we treat an elementary example that satisfies the assumptions made in the previous two sections. 

Let $V = H_0^1([0,1])$,$H=L^2([0,1])$, $D=\{-1,1\}$, $U=[-1,1]$. $V$ is a Hilbert space with inner product 
\[
(v,w)_V := \int_0^1 v(x)w(x) + v'(x)w'(x) \mathrm{d}x. 
\]

\noindent We consider the following PDE for the deterministic evolution between jumps

\begin{equation*}
\frac{\partial}{\partial t }v(t,x) = \Delta v(t,x) + (d+u)v(t,x),
\end{equation*}
with Dirichlet boundary conditions.
We define the jump rate function for $(v,u)\in H\times U$ by 

\begin{equation*}
\lambda_1(v,u) = \frac{1}{e^{-||v||^2}+1} + u^2, \qquad \lambda_{-1}(v,u) = e^{-\frac{1}{||v||^2+1}} + u^2,
\end{equation*}
and the transition measure by $\mathcal{Q}(\{-1\}|v,1,u) = 1$, and $\mathcal{Q}(\{1\}|v,-1,u) = 1.$

Finally, we consider a quadratic cost function $
c(v,u) = K||V_{\mathrm{ref}}-v||^2 + u^2$, where $V_{\mathrm{ref}}\in D(\Delta)$ is a reference signal that we want to approach.

\begin{lemma} The PDMP defined above admits the continuous bounding function 

\begin{equation}
b(v) := ||V_{\mathrm{ref}}||_H^2 + ||v||_H^2 + 1.
\end{equation}

Furthermore, the value function of the optimal control problem is continuous and there exists an optimal ordinary control strategy.

\end{lemma}

\begin{proof}
The proof consists in verifying that all assumptions of Theorem \ref{TheoExistenceOrdinaryOptimalControl} are satisfied. Assumptions \ref{hyp:Q}, \ref{hyp:cost} and \ref{HypOrdinaryControlExists} are straightforward. For $(v,u)\in H\times U$, $1/2\leq\lambda_1(v,u) \leq 2$ and  $e^{-1} \leq \lambda_1(v,u) \leq 2$.
The continuity in the variable $u$ is straightforward and the locally Lipschitz continuity comes from the fact that the functions $
v\to 1/({e^{-||v||^2}+1})$, and $v \to e^{-\beta(v)}$, with $\beta(v):=1/(||v||^2+1)$, are Fréchet differentiable with derivatives $v\to 2(v,\cdot)_H/(e^{-||v||^2}+1)^2$, and $v \to 2(v,\cdot)_H\beta^2(v)e^{-\beta(v)}.$

\noindent $-\Delta v : w\in V \rightarrow \int_0^1 v'(x)w'(x)\mathrm{d}x$ so that $-\Delta : V\to V^*$ is linear. Let $(v,w)\in V^2$.

\[
\langle -\Delta (v-w) , v-w \rangle  = \int_0^1 ((v-w)'(x))^2\mathrm{d}x \geq 0.
\]

\[
|\langle -\Delta v , w \rangle|^2 = |\int_0^1 v'(x)w'(x)\mathrm{d}x|^2 \leq \int_0^1 (v'(x))^2\mathrm{d}x \int_0^1 (w'(x))^2\mathrm{d}x \leq ||v||^2_V ||w||^2_V,
\]

\noindent and so $||-\Delta v||_{V^*} \leq ||v||_V$. $\langle -\Delta v , v \rangle = \int_0^1 (v'(x))^2\mathrm{d}x \geq C'||v||^2_V$, for some constant $C'>0$, by the Poincaré inequality.

Now, define for $k\in\N^*$, $f_k(\cdot) := \sqrt{2}\sin(k\pi\cdot)$, a Hilbert base of $H$. 
%
%
%
%
On $H$, $S(t)$ is the diagonal operator 

\[
S(t)v = \sum_{k\geq 1} e^{-(k\pi)^2t} (v,f_k)_Hf_k.
\]

\noindent For $t>0$, $S(t)$ is a contracting Hilbert-Schmidt operator. 

\noindent For $(v,w,u)\in H^2\times U$, $f_d(v,u) = (d+u)v$ and

\[
||f_d(v,u)-f_d(w,u)||_H \leq 2 ||v-w||_H, \qquad ||f_d(v,u)||_H \leq 2||v||_H.
\]

This means that for every $z=(v,d,h)\in\Upsilon$, $\gamma\in\mathcal{R}([0,T],U)$ and $t\in [0,T]$, $||\phi_t^{\gamma}(z) ||_H \leq e^{2T}||v||_H$.

\end{proof}

\section{Application to the model in Optogenetics}\label{OptoSection}
\subsection{Proof of Theorem \ref{ChR2TheoremExistenceOptimalControl}}\label{ProofTheoSubsec}

We begin this section by making some comments on Definition \ref{stoInfiniteHH}. In (\ref{PDEHHsto}), $C_m>0$ is the membrane capacitance and $V_-$ and $V_+$ are constants defined by $V_-:=\min\{V_{Na},V_K,V_L,\\ V_{ChR2}\}$ and $V_+:=\max\{V_{Na},V_K,V_L,V_{ChR2}\}$. They represent the physiological domain of our process. In (\ref{fdHH}), the constants $g_x > 0$ are the normalized conductances of the channels of type $x$ and $V_x\in\R$ are the driving potentials of the channels. The constant $\rho>0$ is the relative conductance between the open states of the $ChR2$. For a matter of coherence with the theoretical framework presented in the paper, we will prove Theorem \ref{ChR2TheoremExistenceOptimalControl} for the mollification of the model that we define now. This model is very close to the one of Definition \ref{stoInfiniteHH}. It is obtained by replacing the Dirac masses $\delta_z$ by their mollifications $\xi^N_z$ that are defined as follows. Let $\varphi$ be the function defined on $\mathbb{R}$ by

\begin{equation}
\varphi(x):= \left \{
\begin{aligned}
& Ce^{\frac{1}{x^2-1}}, & \text{ if } |x|< 1,\\
& 0 , & \text{ if } |x| \geq 1,
\end{aligned}
\right .
\end{equation}		
with $C:=\left(\int_{-1}^1 \exp\left(\frac{1}{x^2-1}\right) \mathrm{d}x\right)^{-1}$ such that $\int_{\mathbb{R}}\varphi(x)\mathrm{d}x = 1$. 

Now, let $U_N:=\left(\frac{1}{2N},1-\frac{1}{2N}\right)$ and $\varphi_N(x) := 2N \varphi(2Nx)$ for $x\in\mathbb{R}$. For $z\in I_N$, the $N^{\mathrm{th}}$ mollified Dirac mass $\xi^N_z$ at $z$ is defined for $x\in [0,1]$ by

\begin{equation}
\xi^N_z(x):= \left \{
\begin{aligned}
& \varphi_N(x-z), & \text{ if } x\in U_N\\
& 0 , & \text{ if } x\in [0,1]\setminus U_N.
\end{aligned}
\right .
\end{equation}	

For all $z\in I_N, \xi^N_z \in C^{\infty}([0,1])$ and $\xi^N_z \rightarrow \delta_z$ almost everywhere in $[0,1]$ as $N\rightarrow +\infty$, so that $(\xi^N_z,\phi)_H \to \phi(z)$, as $N\to \infty$ for every $\phi\in C(I,\R)$. The expressions $v(i/N)$ in Definition \ref{stoInfiniteHH}  are also replaced by $(\xi^N_{i/N},v)_H$. The decision to use the mollified Dirac mass over the Dirac mass can be motivated by two main reasons. First of all, as mentioned in \cite{Riedler}, the concentration of ions is homogeneous in a spatially extended domain around an open channel so the current is modeled as being present not only at the point of a channel, but in a neighborhood of it. Second, the smooth mollified Dirac mass leads to smooth solutions of the PDE and we need at least continuity of the flow. Nevertheless, the results of Theorem \ref{ChR2TheoremExistenceOptimalControl} remain valid with the Dirac masses and we refer the reader to Section \ref{VariationSection}.

The following lemma is a direct consequence of \cite[Proposition 7]{Riedler} and will be very important for the model to fall within the theoretical framework of the previous sections.

\begin{lemma}\label{modelChR2Bounded}
For every $y_0\in V$ with $y_0(x) \in [V_-,V_+]$ for all $x\in I$, the solution $y$ of (\ref{PDEHHsto}) is such that for $t\in [0,T]$,

\[
V_- \leq y(t,x) \leq V_+, \quad \forall x\in I. 
\]
\end{lemma}

Physiologically speaking, we are only interested in the domain $[V_-,V_+]$. Since Lemma \ref{modelChR2Bounded} shows that this domain is invariant for the controlled PDMP, we can modify the characteristics of the PDMP outside the domain $[V_-,V_+]$ without changing its dynamics. We will do so for the rate functions $\sigma_{x,y}$ of Table \ref{jumprates}. From now on, consider a compact set $K$ containing the closed ball of $H$, centered in zero and with radius $\max(V_-,V_+)$. We will rewrite $\sigma_{x,y}$ the quantities modified outside $K$ such that they all become bounded functions. This modification will enable assumption \ref{hyp:lambda}1. to be verified. 

The next lemma shows that the stochastic controlled infinite-dimensional Hodgkin-Huxley-ChR2 model defines a controlled infinite-dimensional PDMP as defined in Definition \ref{ControlStrat} and that Theorem \ref{THEOGENERATOR} applies.

\begin{lemma}\label{HHChR2Bounded}
For $N\in N^*$, the $N^{\mathrm{th}}$ stochastic controlled infinite-dimensional Hodgkin-Huxley-ChR2 model satisfies assumptions \ref{hyp:lambda}, \ref{hyp:Q}, \ref{hyp:L} and \ref{hyp:f}. Moreover, for any control strategy $\alpha\in\mathcal{A}$, the membrane potential $v^\alpha$ satisfies 

\[
V_- \leq v_t^{\alpha}(x) \leq V_+, \quad \forall (t,x)\in [0,T]\times I. 
\]

\end{lemma}

\begin{proof}
The local Lipschitz continuity of $\lambda_d$ from $H\times Z$ in $\R^+$ comes from the local Lipschitz continuity of all the functions $\sigma_{x,y}$ of Table \ref{THEOGENERATOR} and the inequality $|(\xi^N_z,v)_H-(\xi^N_z,w)_H|\leq 2N||v-w||_H$. By Lemma \ref{modelChR2Bounded}, the modified jump rates are bounded. Since they are positive, they are bounded away from zero, and then, Assumption \ref{hyp:lambda} is satisfied. Assumption \ref{hyp:Q} is also easily satisfied. We showed in Section \ref{elementaryExample} that \ref{hyp:L} is satisfied. As for $f_d$, the function does not depend on the control variable and is continuous from $H$ to $H$. For $d\in D$ and $(y_1,y_2)\in H^2$, 

\begin{align*}
f_d(y_1) - f_d(y_2) &= \frac{1}{N} \sum_{i\in I_N} \Big( g_K\mathbf{1}_{\{d_i=n_4\}}+ g_{Na} \mathbf{1}_{\{d_i=m_3h_1\}}\\
& \qquad + g_{ChR2}( \mathbf{1}_{\{d_i=O_1\}}+\rho  \mathbf{1}_{\{d_i=O_2\}})+ g_L\Big)(\xi^N_{\frac{i}{N}},y_2-y_1)_H \xi^N_{\frac{i}{N}}.
\end{align*}

We then get

\[
||f_d(y_1) - f_d(y_2)||_H \leq 4N^2 ( g_K + g_{Na}+ g_{ChR2}( 1+\rho)+ g_L)||(y_2-y_1)||_H .
\]

Finally, since the continuous component $v_t^{\alpha}$ of the PDMP does not jump, the bounds are a direct consequence of Lemma \ref{modelChR2Bounded}.

\end{proof}

\begin{proof}[Proof of Theorem \ref{ChR2TheoremExistenceOptimalControl}]
In Lemma \ref{HHChR2Bounded} we already showed that assumptions  \ref{hyp:lambda}, \ref{hyp:Q}, \ref{hyp:L} and \ref{hyp:f} are satisfied. The cost function $c$ is convex in the control variable and norm quadratic on $H\times Z$. The flow does not depend on the control. The rate function $\lambda$ is linear in the control. the function $\lambda \mathcal{Q}$ is also linear in the control. We conclude that all the assumptions of Theorem \ref{TheoremExistenceOptimalControl}
are satisfied and that an optimal ordinary strategy can be retrieved.

\end{proof}

We end this section with an important remark that significantly extends the scope of this example. Up to now, we only considered stationary reference signals but nonautonomous ones can be studied as well, as long as they feature some properties. Indeed, it is only a matter of incorporating the signal reference $V_{\mathrm{ref}}\in C([0,T],H)$ in the process by adding a variable to the PDMP. Instead of considering $H$ as the initial state space for the continuous component, we consider $\tilde{H} := H\times H$.

This way, the part on the control problem is not impacted at all and we consider the continuous cost function $\tilde{c}$ defined for $(v,\bar{v},u)\in \tilde{H}\times U$ by 

\begin{equation}\label{HHcostFunctionTilde}
\tilde{c}(v,\bar{v},u) = \kappa ||v-\bar{v}||_H^2 + u + c_{\mathrm{min}},
\end{equation}

the result and proof of lemma \ref{ChR2TheoremExistenceOptimalControl} remaining unchanged with the continuous bounding function defined for $v\in H$ by 

\begin{equation*}
b(v) := \left\{
\begin{aligned}
& \kappa M_3^2 + \kappa\sup_{t\in [0,T]}||V_{ref}(t)||^2_H + u_{max},&\text{if  } ||v||_H\leq M_3,\\
&  \kappa||v||^2_H + \kappa\sup_{t\in [0,T]}||V_{ref}(t)||^2_H + u_{max},& \text{if  } ||v||_H > M_3.
\end{aligned}
\right.
\end{equation*}

In the next section, we present some variants of the model and the corresponding results in terms of optimal control.

\begin{table}[!h]
\begin{center}
\caption{Expression of the individual jump rate functions.}
\label{jumprates}
\begin{tabular}{l}
\hrulefill
\\
\underline{ In $D_1 = \{n_0,n_1,n_2,n_3,n_4\}$} : \\ [0.1cm]
\begin{tabular}{llll}
$\sigma_{n_0,n_1}(v,u) = 4\alpha_n(v)$, &$\sigma_{n_1,n_2}(v,u) = 3\alpha_n(v)$, &$\sigma_{n_2,n_3}(v,u) = 2\alpha_n(v)$, &$\sigma_{n_3,n_4}(v,u) = \alpha_n(v)$\\
$\sigma_{n_4,n_3}(v,u) = 4\beta_n(v)$, &$\sigma_{n_3,n_2}(v,u) = 3\beta_n(v)$, &$\sigma_{n_2,n_1}(v,u) = 2\beta_n(v)$, &$\sigma_{n_1,n_0}(v,u) = \beta_n(v)$.\\
\end{tabular}

\\[0.5cm]


\underline{In $D_2 = \{m_0h_1,m_1h_1,m_2h_1,m_3h_1,m_0h_0,m_1h_0,m_2h_0,m_3h_0\}$} :\\[0.1cm]
\begin{tabular}{ll}
$\sigma_{m_0h_1,m_1h_1}(v,u)=\sigma_{m_0h_0,m_1h_0}(v,u) = 3\alpha_m(v)$, &$\sigma_{m_1h_1,m_2h_1}(v,u) =\sigma_{m_1h_0,m_2h_0}(v,u) = 2\alpha_m(v)$,\\ 
$\sigma_{m_2h_1,m_3h_1}(v,u) = \sigma_{m_2h_0,m_3h_0}(v,u) =  \alpha_m(v)$, 
&$\sigma_{m_3h_1,m_2h_1}(v,u) = \sigma_{m_3h_0,m_2h_0}(v,u) = 3\beta_m(v)$,\\
 $\sigma_{m_2h_1,m_1h_1}(v,u) =\sigma_{m_2h_0,m_1h_0}(v,u) = 2\beta_m(v)$, &$\sigma_{m_1h_1,m_0h_1}(v,u) = \sigma_{m_1h_0,m_0h_0}(v,u) = \beta_m(v)$.\\
\end{tabular}

\\[0.7cm]

\underline{In $D_{ChR2} = \{o_1,o_2,c_1,c_2\}$} :\\[0.1cm]

\begin{tabular}{llll}
$\sigma_{c_1,o_1}(v,u) = \varepsilon_1 u $, &$\sigma_{o_1,c_1}(v,u) = K_{d1}$, &$\sigma_{o_1,o_2}(v,u) = e_{12} $, &$\sigma_{o_2,o_1}(v,u) = e_{21}$\\
$\sigma_{o_2,c_2}(v,u) = K_{d2} $, &$\sigma_{c_2,o_2}(v,u) = \varepsilon_2 u $, &$\sigma_{c_2,c_1}(v,u) = K_r$.&\\
\end{tabular}

\\[0.5cm]

\hrulefill

\\

\begin{tabular}{ll}
$\alpha_n(v)=\frac{0.1-0.01v}{e^{1-0.1v}-1}$, &$\beta_n(v)=0.125e^{-\frac{v}{80}}$,\\[0.3cm]
$\alpha_m(v)=\frac{2.5-0.1v}{e^{2.5-0.1v}-1}$, &$\beta_m(v)=4e^{-\frac{v}{18}}$,\\[0.3cm]
$\alpha_h(v)=0.07e^{-\frac{v}{20}}$, &$\beta_h(v)=\frac{1}{e^{3-0.1v}+1}$.\\[0.3cm]
\end{tabular}

\\[1cm]

\hrulefill
\end{tabular}
\end{center}
\end{table}

\subsection{Variants of the model}\label{VariationSection}

We begin this section by giving arguments showing that the results of Theorem \ref{TheoremExistenceOptimalControl} remain valid for the model of Definition \ref{stoInfiniteHH}, which does not exactly fits into our theoretical framework. Then, the variations we present concern the model of ChR2, the addition of other light-sensitive ionic channels, the way the control acts on the three local characteristics and the control space. The optimal control problem itself will remain unchanged. First of all, let us mention that since the model of Definition \ref{stoInfiniteHH} satisfies the convexity conditions \ref{HypOrdinaryControlExists}, the theoretical part on relaxed controls is not necessary for this model. Nevertheless, the model of ChR2 presented on Figure \ref{ChR2figure} is only one among several others, some of which do not enjoy a linear, or even concave, rate function $\lambda$. For those models, that we present next, assumption \ref{HypOrdinaryControlExists} fails and the relaxed controls are essential. 

We will not present them here, but the previous results for the Hodgkin-Huxley model remain straightforwardly unchanged for other neuron models such as the FitzHugh-Nagumo model or the Morris-Lecar model.

\paragraph{Optimal control for the original model}\hspace*{0pt}\\

In the original model, the function $f_d$ is defined from $V$ to $V^*$. Nevertheless, the semigroup of the Laplacian regularizes Dirac masses (see \cite[Lemma 3.1]{Austin}) and the uniform bound in Theorem \ref{THEOGENERATOR} is in fact valid in $V$, the solution belonging to $C([0,T],V)$. This is all we need since the control does not act on the PDE. This is why the domain of our process is $V\times D_N$ and not just $H\times D_N$, and all computations of the proofs of the previous sections can be done in the Hilbert space $V$. From this consideration, and using the continuous embedding of $H_0^1(I)$ in $C_0(I)$ we can justify the local Lipschitz continuity of $\lambda_d$ from $V\times Z$ in $\R^+$. Indeed, it comes from the local Lipschitz continuity of all functions $\sigma_{x,y}$ of Table \ref{jumprates} and from the inequality 

\[
|v(\frac{i}{N})-w(\frac{i}{N})| \leq \sup_{x\in I} |v(x)-w(x)| \leq C ||v-w||_V.
\]

Finally, \cite[Proposition 5]{Riedler} states that the bounds of Lemma \ref{HHChR2Bounded} remain valid with Dirac masses. 

\paragraph{Modifications of the ChR2 model}\hspace*{0pt}\\

We already mentioned the paper of Nikolic and al. \cite{Nikolic} in which a three states model is presented. It is a somehow simpler model that the four states model of Figure \ref{ChR2figure} but it gives good qualitative results on the photocurrents produced by the ChR2.  In first approximation the model can be considered to depend linearly in the control as seen on Figure \ref{ChR23StatesLinear}.

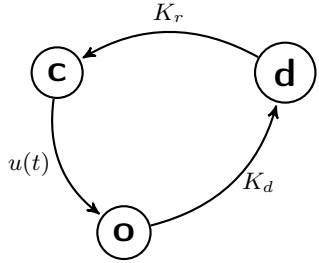
\begin{figure}[!h]
\begin{center}
\begin{tikzpicture}[->,>=stealth',shorten >=1pt,auto,node distance=3cm,
  thick,main node/.style={circle,draw,font=\sffamily\Large\bfseries}]

  \node[main node] (1) {c};
  \node[main node] (2) [right of=1] {d};
  \node[main node] (3) [below left of=2] {o};

  \path[every node/.style={font=\sffamily\small}]
    (1) edge [bend right] node [left] {$u(t)$} (3)
    (2) edge  [bend right] node [above]  {$K_r$} (1)
    (3) edge [bend right] node [right] {$K_d$} (2);

\end{tikzpicture}
\end{center}
\caption{Simplified ChR2 three states model}
\label{ChR23StatesLinear}
\end{figure}

This model features one open state $o$ and two closed states, one light-adapted $d$ and one dark-adapted $c$. This model would lead to the same type of model as in the previous Section. In fact, the time constants $1/K_d$ and $1/K_r$ are also light dependent with a dependence in $\log(u)$. The corresponding model is represented on Figure \ref{ChR23StatesLog} below

\begin{figure}[!h]
\begin{center}
\begin{tikzpicture}[->,>=stealth',shorten >=1pt,auto,node distance=3cm,
  thick,main node/.style={circle,draw,font=\sffamily\Large\bfseries}]

  \node[main node] (1) {c};
  \node[main node] (2) [right of=1] {d};
  \node[main node] (3) [below left of=2] {o};

  \path[every node/.style={font=\sffamily\small}]
    (1) edge [bend right] node [left] {$c_1u(t)$} (3)
    (2) edge  [bend right] node [above]  {$K_r + c_2\log(u)$} (1)
    (3) edge [bend right] node [right] {$\frac{1}{\tau_d-\log(u)}$} (2);

\end{tikzpicture}
\end{center}
\caption{ChR2 three states model}
\label{ChR23StatesLog}
\end{figure}
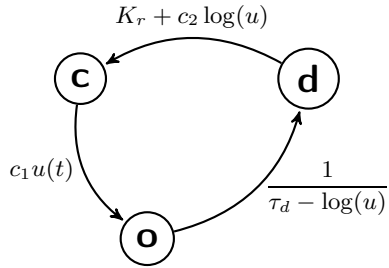

Some mathematical comments are needed here. On Figure \ref{ChR23StatesLog}, the control $u$ represents the light intensity and $c_1$, $c_2$, $K_r$ and $\tau_d$ are positive constants. This model of ChR2 is experimentally accurate for intensities between $10^8$ and $10^{10}$ $\mu\mathrm{m}^2\cdot\mathrm{s}^{-1}$ approximately. We would then consider $U:=[0,u_{max}]$ with $u_{max}\simeq 10^{10}$ $\mu\mathrm{m}^2\cdot\mathrm{s}^{-1}$. Furthermore, 

\[
\lim_{u\to 0} K_r + c_2\log(u) = -\infty, \qquad \lim_{u\to 0} \frac{1}{\tau_d-\log(u)} = 0.
\]

The first limit is not physical since rate jumps between states are positive numbers. The second limit is not physical either because it would mean that, in the dark, the proteins are trapped in the open state $o$, which is not the case. In the dark, when $u=0$, the jump rates corresponding to the transition $o\to d$ and $d\to c$ are positive constants. For this reason, the functions $\sigma_{o,d}$ and $\sigma_{d,c}$ should be smooth functions such that they are equal to the rates of Figure \ref{ChR23StatesLog} for large intensities, but still with $\tau_d-\log(u) >0$, and converge to $K_d^{\mathrm{dark}}>0$ and $K_r^{\mathrm{dark}}>0$ respectively, when $u$ goes to $0$. The resulting rate function $\lambda$ is not concave and thus does not satisfy assumption \ref{HypOrdinaryControlExists} anymore. We can only affirm the existence of optimal relaxed strategies.

The four states model of Figure \ref{ChR2figure} is also an approximation of a more accurate model that we represent on Figure \ref{ChR24StatesLog} below. The transition rates can depend on either the membrane potential $v$ or the irradiance $u$, which is the control variable. The details of the model and the numerical constants can be found in \cite{ChR2Modelling}. Note that the model of Figure \ref{ChR24StatesLog} is already an approximation of the model in \cite{ChR2Modelling} because the full model in \cite{ChR2Modelling} would not lead to a Markovian behavior for the ChR2 (the transition rates would depend on the time elapsed since the light was switched on).

\begin{figure}[!t]
\begin{center}
\begin{tikzpicture}[->,>=stealth',shorten >=1pt,auto,node distance=3cm,
  thick,main node/.style={circle,draw,font=\sffamily\Large\bfseries}]

  \node[main node] (1) {$o_1$};
  \node[main node] (2) [right of=1] {$o_2$};
  \node[main node] (3) [below of=2] {$c_2$};
  \node[main node] (4) [below of=1] {$c_1$};

  \path[every node/.style={font=\sffamily\small}]
    (1) edge [bend left] node [right] {$K_{d1}(v)$} (4)
        edge [bend right] node [above] {$e_{12}(u)$} (2)
    (2) edge [bend right] node [above] {$e_{21}(u)$} (1)
        edge [bend right] node[right] {$K_{d2}$} (3)
    (3) edge [bend right] node [right] {$K_{a2} u$} (2)
        edge node {$K_r(v)$} (4)
    (4)edge [bend left] node[left] {$K_{a1} u$} (1);

\end{tikzpicture}
\caption{ChR2 channel : $K_{a1}$, $K_{a2}$, and $K_{d2}$ are positive constants.}\label{ChR24StatesLog}
\end{center}
\end{figure}
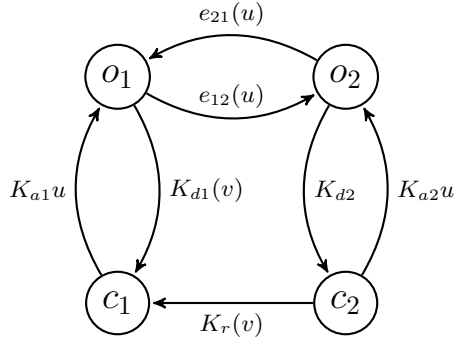

\begin{align*}
K_{d1}(v)&=K_{d1}^{(1)}-K_{d1}^{(2)}\tanh((v+20)/20),\\
e_{12}(u)&=e_{12d}+c_1\ln(1+u/c),\\
e_{21}(u)&=e_{21d}+c_2\ln(1+u/c),\\
K_r(v)&=K_r^{(1)}\exp(-K_r^{(2)}v),\\
\end{align*}

 with $K_{d1}^{(1)}$, $K_{d1}^{(2)}$, $e_{12d}$, $e_{21d}$, $c$, $c_1$ and $c_2$ positive constants. As for the model of Figure \ref{ChR23StatesLog}, the mathematical definition of the function $\sigma_{o_1,c_1}$ should be such that it is a positive smooth function and equals $K_{d1}(v)$ in some subset of the physiological domain $[V_-,V_+]$. The resulting rate function $\lambda$ will be concave but the function $\lambda\mathcal{Q}$ will not be convex (it will be concave as well). Hence, Assumption \ref{HypOrdinaryControlExists} is not satisfied.

\paragraph{Addition of other light-sensitive ion channels}\hspace*{0pt}\\

Channelrhodopsin-2 has a promoting role in eliciting action potentials. There also exists a chlorine pump, called Halorhodopsin (NpHR), that has an inhibitory action. NpHR can be used along  with ChR2 to obtain a control in both directions. Its modelisation as a multistate model was considered in \cite{NikolicNpHr}. The transition rates between the different states have the same shape that the ones of the ChR2 and the same simplifications are possible. This new light-sensitive channel can be easily incorporated in our stochastic model and we can state existence of optimal relaxed and/or ordinary control strategies depending on the level of complexity of the NpHR model we consider. It is here important to remark that since the two ionic channels do not react to the same wavelength of the light, the resulting control variable would be two-dimensional with values in $[0,u_{max}]^2$. This would not change the qualitative results of the previous sections.

\paragraph{Modification of the way the control acts on the local characteristics}\hspace*{0pt}\\

Up to now, the control acts only on the rate function, and also on the measure transition via its special definition from the rate function. Nevertheless, we can present here a modification of the model where the control acts linearly on the PDE. This modification amounts to considering that the control variable is directly the gating variable of the ChR2. Indeed, we show in \cite{MinimalTime} that the optimal control of the deterministic counterpart of the stochastic Hodgkin-Huxley-ChR2 model, in finite dimension and with the three states ChR2 model of Figure \ref{ChR23StatesLinear}, is closely linked to the optimal control of 

\begin{equation*}
\left\{
\begin{aligned}
\frac{\mathrm{d}V}{\mathrm{d}t} & =  g_Kn^4(t)(V_K-V(t)) + g_{Na}m^3(t)h(t)(V_{Na}-V(t))\\ 
& \qquad +g_{ChR2}u(t)(V_{ChR2}-V(t)) + g_L(V_L-V(t)),\\ 
\frac{\mathrm{d}n}{\mathrm{d}t} & =  \alpha_n(V(t))(1-n(t)) - \beta_n(V(t)) n(t),\\
\frac{\mathrm{d}m}{\mathrm{d}t} & =  \alpha_m(V(t))(1-m(t)) - \beta_m(V(t)) m(t),\\
\frac{\mathrm{d}h}{\mathrm{d}t} & =  \alpha_h(V(t))(1-h(t)) - \beta_h(V(t)) h(t),\\
\end{aligned}
\right.
\end{equation*}

where the control variable is the former gating variable $o$. Now the stochastic counterpart of the last model is such that the function $f_d$ is now linear in the control and the rate function $\lambda$ and the transition measure function $\mathcal{Q}$ do not depend on the control any more. Finally, by adding NpHR channels to this model, we would obtain a fully controlled infinite-dimensional PDMP in the sense that the control would then act on the three local characteristics of the PDMP. Depending on the model of NpHR chosen, we would obtain relaxed or ordinary optimal control strategy.

\paragraph{Modification of the control space}\hspace*{0pt}\\

In all models discussed previously, the control has no spatial dependence. Any light-stimulation device, such as a laser, has a spatial resolution and it is possible that we do not want or cannot stimulate the entire axon. For this reason, spatial dependence of the control should be considered. Now, as long as the control space remains a compact Polish space, spatial dependence of the control could be considered. We propose here a control space defined as a subspace of the Skorohod space $\mathbb{D}$, constituted of the \textit{càdlàg} functions from $[0,1]$ to $\R$. This control space represents the aggregation of multiple laser beams that can be switched on and off. Suppose that each of theses beams produces on the axon a disc of light of diameter $r > 0$ that we call spatial resolution of the light. For an axon represented by the segment $[0,1]$, $r$ is exactly the spatial domain illuminated. We consider now two possibilities for the control space. Suppose first that the spatial resolution is fixed and define $p:=\lfloor\frac{1}{r}\rfloor$ and 

\[
\mathcal{U} := \{ u : [0,1] \rightarrow [0,u_{max}] \mid u \text{ is constant on } [i/p,(i+1)/p), i=0,..,p-1, u(1) = u((p-1)/p)\}.
\]

\begin{lemma}
$\mathcal{U}$ is a compact subset of $\mathbb{D}$.
\end{lemma}

\begin{proof} We tackle this proof by remarking that $\mathcal{U}$ is in bijection with the finite dimensional compact space $[0,u_{max}]^p$.
\end{proof}

 In this case, the introduction of the space $\mathbb{D}$ was quite artificial since the control space remains finite-dimensional. Nevertheless, the Skorohod space will be very useful for the other control space. Suppose now that the spatial resolution of the laser can evolve in $[r_{min},r_{max}]$ with $r_{min},r_{max} > 0$. Let $p\in\N^*$ the number of lasers used and define 
 
 \begin{align*}
\tilde{\mathcal{U}} := \{ u : [0,1] \rightarrow [0,u_{max}] & \mid \exists \{x_i\}_{0\leq i\leq p} \text{ subdivision of }[0,1],\\
& \qquad u \text{ is constant on } [x_i,x_i+1), i=0,..,p-1,\\
& \qquad u(1) = u(x_{p-1})\}.
\end{align*}

Now $\tilde{\mathcal{U}}$ is infinite-dimensional and the Skorohod space allows us to use the characterization of compact subsets of $\mathbb{D}$.

\begin{lemma}
$\tilde{\mathcal{U}}$ is a compact subset of $\mathbb{D}$.
\end{lemma}

\begin{proof}
For this proof, we need to introduce some notation and a critera of compactness in $\mathbb{D}$. A complete treatment of the space $\mathbb{D}$ can be found in \cite{Billingsley}.

Let $u\in \mathbb{D}$ and $\{x_i\}_{0\leq i\leq n}$ a subdivision of $[0,1]$, $n\in \N^*$. We define, for $i\in\{0,..,n-1\}$,

\[
w_u([x_i,x_{i+1})) := \sup_{x,y\in [x_i,x_{i+1})} |u(x)-u(y)|,
\]
and for $\delta >0$,

\[
w'_u(\delta) := \inf_{\{x_i\}} \max_{0\leq i < n}w_u([x_i,x_{i+1})),
\]
the infimum being taken on all the subdivisions $\{x_i\}_{0\leq i\leq n}$ of $[0,1]$ such that $x_{i+1}-x_i > \delta$ for all $i\in\{0,..,n-1\}$.
Now since $\tilde{\mathcal{U}}$ is obviously bounded in $\mathbb{D}$, from \cite[Theorem 14.3]{Billingsley}, we need to show that 

\[
\lim_{\delta\to 0}\sup_{u\in\tilde{\mathcal{U}}} w'_u(\delta) = 0.
\]

Let $\delta > 0$ with $\delta < r_{min}$ and $u\in \tilde{\mathcal{U}}$. There exists as subdivision $\{x_i\}_{0\leq i\leq p}$ of $[0,1]$ such that for every $i\in \{0,..,p-1\}$, $u$ is constant on $[x_i,x_{i+1})$ and $x_{i+1}-x_i > \delta$. Thus $w'_u(\delta)=0$ which ends the proof.

\end{proof}

With either $\mathcal{U}$ or $\tilde{\mathcal{U}}$ as the control space, the stochastic controlled infinite-dimensional Hodgkin-Huxley-ChR2 model admits an optimal ordinary control strategy.

\begin{appendices}
\section{Construction of $X^\alpha$ by iteration}\label{iteration}

Let $\alpha\in\mathcal{A}$ and let $x:=(v,d,\tau,h,\nu)\in\Xi^{\alpha}$ with $z:=(\nu,d,h)\in\Upsilon$. The existence of the probability $\mathbb{P}_x^{\alpha}$ below is the object of the next section where Theorem \ref{THEOGENERATOR} is proved. 
\begin{itemize}
\item Let $T_1$ be the time of the first jump of $(X^{\alpha}_t)$. With the notations of Proposition \ref{flowproperty}, the law of $T_1$ is defined by its survival function given for all $t>0$ by
\begin{equation*}
\mathbb{P}_x^{\alpha}(T_1>t) = \exp\left(-\int_0^t \lambda_{d}\Big( \phi_s^{\alpha}(x),\alpha(\nu_s,d_s,h_s)(\tau_s)\Big)\mathrm{d}s\right).
\end{equation*}

\bigskip

\item For $t<T_1$, $X_t^{\alpha}$ solves (\ref{PDEenlarged}) starting from $x$ namely $
(v_t, d_t , \tau_t,h_t, \nu_t)=( \phi_t^{\alpha}(x), d, \tau+t,h,\nu).$

\item When a jump occurs at time $T_1$, conditionally to $T_1$, $X_{T_1}^{\alpha}$ is a random variable distributed according to a measure $\hat{\mathcal{Q}}$ on $(\Xi,\mathcal{B}(\Xi))$, itself defined by a measure $\mathcal{Q}$ on $(D,\mathcal{B}(D))$. The target state $d_1$ of the discrete variable is a random variable distributed according to the measure $\mathcal{Q}(\cdot|\phi_{T_1}^{\alpha}(x),d_{T_1^-},\alpha(\nu_{T_1^-},d_{T_1^-},h_{T_1^-})(\tau_{T_1^-}))$ such that for all $B\in\mathcal{B}(D)$, 

\begin{align}
\hat{\mathcal{Q}}\Big(\{\phi_{T_1}^{\alpha}(x)\}\times B&\times \{0\} \times \{h+\tau_{T_1^-}\}\times\{\phi_{T_1}^{\alpha}(x)\}|\phi_{T_1}^{\alpha}(x),d_{T_1^-},\tau_{T_1^-},h_{T_1^-},\nu_{T_1^-},\alpha(T_1^-)\Big)\nonumber\\
&  = \mathcal{Q}\Big(B|\phi_{T_1}^{\alpha}(s),d,\alpha(\nu,d,h)(\tau+T_1)\Big),\nonumber
\end{align}
where we use the notation $\alpha(T_1^-)=\alpha(d_{T_1^-},\tau_{T_1^-},h_{T_1^-},\nu_{T_1^-})$. This equality means that the variables $v$ and $\nu$ do not jump at time $T_1$, and the variables $\tau$ and $h$ jump in a deterministic way to $\{0\}$ and $\{h+\tau_{T_1^-}\}$ respectively.
\item The construction iterates after time $T_1$ with the new starting point $(v_{T_1},d_{T_1},0,h+T_1,v_{T_1}$).
\end{itemize}

Formally the expressions of the jump rate and the transition measures on $\Xi$ are

\begin{align*}
\lambda(x,u) &:= \lambda_d(v,u),\\
\hat{\mathcal{Q}}\Big(F\times B \times E \times G \times J | x,u\Big) &:=\mathbf{1}_{F\times E\times G\times J}(v,0,h+\tau,\nu)\mathcal{Q}\Big(B|v,d,u\Big),
\end{align*}
with $F\times B \times E \times G \times J\in\mathcal{B}(\Xi)$, $u\in U$ and $x:=(v,d,\tau,h,\nu)\in\Xi$.

\section{Proof of Theorem \ref{THEOGENERATOR}}\label{proofTheorem1}

There are two filtered spaces on which we can define the enlarged process $(X^\alpha)$ of Definition \ref{ControlStrat}. They are linked by the one-to-one correspondence between the PDMP $(X^\alpha)$ and the included jump process $(Z^\alpha)$ that we define now. We then introduce both spaces since each one of them is relevant to prove useful properties.

\medskip
Given the sample path $(X_s^{\alpha},s\leq T)$ such that $X_0^{\alpha}:=(v,d,\tau,h,\nu)\in\Xi^{\alpha}$, the jump times $T_k$ of $X^{\alpha}$ can be retrieved by the formula 
\[
\{T_k, k=1,\dots,n\} = \{s\in(0,T]| h_{s}\neq h_{s^-}\}.
\]
Moreover we can associate to $X^{\alpha}$ a pure jump process $(Z_t^{\alpha})_{t\geq 0}$ taking values in $\Upsilon$ in a one-to-one correspondence as follows, 

\begin{equation}\label{associatedJumpProc}
Z_t^{\alpha} := (\nu_{T_k},d_{T_k},T_k), \qquad T_k\leq t<T_{k+1}.
\end{equation}

Conversely, given the sample path of $Z^{\alpha}$ on $[0,T]$ starting from $Z_0^{\alpha}=(\nu_0^{Z},d_0^{Z},T_0^{Z}) $, we can recover the path of $X^{\alpha}$ on $[0,T]$. Denote $Z_t^{\alpha}$ as $(\nu_t^{Z},d_t^{Z},T_t^{Z})$ and define $T_0: = T_0^Z$ and $T_k := \inf\{t>T_{k-1}|T_t^Z\neq T_{t^-}^Z\}$. Then

\begin{equation}\label{associatedPDMP}
\left\{
\begin{aligned}
X_t^{\alpha} &= (\phi_t^{\alpha}(Z_0^{\alpha}),d_0^{Z},t,T_0^{Z},\nu_0^{Z}),  & t<T_1,\\
X_t^{\alpha} &= (\phi_{t-T_k}^{\alpha}(Z_{T_k}^{\alpha}),d_{T_k},t-T_k,T_{T_k}^Z,\nu_{T_k}^Z),& T_k\leq t<T_{k+1}.
\end{aligned}
\right.
\end{equation}
Let us note that $ T_{T_k}^Z = T_k$ for all $k\in\mathbb{N}$, and that by construction of the PDMP all jumps are detected since $\mathbb{P}^{\alpha}[T_{k+1}=T_k]=0$. When no confusion is possible, we write, for $\alpha \in\mathcal{A}$ and $n\in\mathbb{N}$, $Z_n=Z^{\alpha}_{T_n}$.


\paragraph{Part 1. The canonical space of jump processes with values in $\Upsilon$.}
The following construction is very classical, see for instance Davis \cite{DavisBook} Appendix A1. We adapt it here to our peculiar process and to the framework of control. Remember that a jump process is defined by
a sequence of inter-arrival times and jump locations
\begin{equation}
\omega = (\gamma_0,s_1,\gamma_1,s_2,\gamma_2,\dots), 
\end{equation}
where $\gamma_0\in \Upsilon$ is the initial position, and for $i\in\mathbb{N}^*$, $s_i$ is the time elapsed between the $(i-1)^{\mathrm{th}}$ and the $i^{\mathrm{th}}$ jump while $\gamma_i$ is the location right after the $i^{\mathrm{th}}$ jump. The jump times $(t_i)_{i\in\mathbb{N}}$ are deduced from the sequence $(s_i)_{i\in\mathbb{N}^*}$ by $t_0=0$ and $t_i=t_{i-1}+s_i$ for $i\in\mathbb{N}^*$ and the jump process $(J_t)_{t\geq0}$ is given by $J_t := \gamma_i$ for $t\in[t_i,t_{i+1})$ and $J_t=\Delta$ for $t\geq t_{\infty}:=\lim_{i\rightarrow \infty} t_i$, $\Delta$ being an extra state, called cemetery. 

Accordingly we introduce $Y^{\Upsilon} := (\mathbb{R}_+\times\Upsilon)\cup\{(\mathbb{R}_+\cup\infty,\Delta)\}$. Let $(Y_i^{\Upsilon})_{i\in\mathbb{N}^*}$ be a sequence of copies of the space $Y^{\Upsilon}$. We define 
$\Omega^{\Upsilon} :=  \Upsilon\times\Pi_{i=1}^{\infty} Y_i^{\Upsilon}$ the canonical space of jump processes with values in $\Upsilon$, endowed with its Borel $\sigma$-algebra $\mathcal{F}^{\Upsilon}$ and the coordinate mappings on $\Omega^{\Upsilon}$ as follows
\begin{equation}
\left\{
\begin{aligned}
S_i :&\quad \Omega^{\Upsilon}& \longrightarrow &\mathbb{R}_+\cup \{\infty\}, \\
&\quad\omega &\longmapsto & S_i(\omega) = s_i,\quad \text{for } i\in \mathbb{N}^*,\\
\Gamma_i :&\quad \Omega^{\Upsilon}& \longrightarrow &\Upsilon\cup\{\Delta\}, \\
&\quad\omega &\longmapsto & \Gamma_i(\omega) = \gamma_i,\quad \text{for } i\in \mathbb{N}.
\end{aligned}
\right.
\end{equation}
We also introduce $\, \omega_i : \Omega^{\Upsilon}\rightarrow \Omega^{\Upsilon}_i$ for $i\in\mathbb{N}^*$, defined by 
\begin{equation*}
\omega_i(\omega):=(\Gamma_0(\omega),S_1(\omega),\Gamma_1(\omega),\dots,S_i(\omega),\Gamma_i(\omega))
\end{equation*}
for $\omega\in\Omega^{\Upsilon}$. Now for $\omega\in\Omega^{\Upsilon}$ and $i\in\mathbb{N}^*$, let 
\begin{align*}
T_0(\omega) &:=0,\\
T_i(\omega) &:=\left\{
\begin{aligned}
&\sum_{k=1}^i S_k(\omega),  &\text{if } S_k(\omega) \neq \infty \text{ and } \Gamma_k(\omega) \neq \Delta, k=1,\dots,i,\\ 
&\infty & \text{if }  S_k(\omega) = \infty \text{ or } \Gamma_k(\omega) = \Delta\text{ for some } k=1,\dots,i,
\end{aligned}
\right.\\
T_{\infty}(\omega) &:= \lim_{i\rightarrow\infty} T_i(\omega).
\end{align*}
and the sample path $(x_t(\omega))_{t\geq0}$ be defined by 
\begin{equation}
x_t(\omega) := \left\{
\begin{aligned}
\Gamma_i(\omega) &\qquad T_i(\omega)\leq t<T_{i+1}(\omega),& \\
\Delta & \qquad t\geq T_{\infty}(\omega).
\end{aligned}
\right.
\end{equation}
A relevant filtration for our problem is the natural filtration of the coordinate process $(x_t)_{t\geq 0}$ on $\Omega^{\Upsilon}$

\begin{equation*}
\mathcal{F}_t^{\Upsilon} := \sigma \{ x_s | s \leq t\},
\end{equation*}
for all $t\in\mathbb{R}_+$. For given starting point $\gamma_0\in\Upsilon$ and control strategy $\alpha\in\mathcal{A}$, a \textit{controlled probability measure}, denoted $\mathbb{P}^{\alpha}_{\gamma_0}$, is defined on $\Omega^{\Upsilon}$ by the specification of a family of controlled conditional distribution functions as follows: $\mu_1$ is a controlled probability measure on $(Y^{\Upsilon},\mathcal{B}(Y^{\Upsilon}))$ or equivalently a measurable mapping from $\mathcal{U}_{ad}([0,T];U)$ to the set of probability measures on $(Y^{\Upsilon},\mathcal{B}(Y^{\Upsilon}))$, such that for all $\alpha\in\mathcal{A}$,
\begin{equation*}
\mu_1(\alpha(\gamma_0);(\{0\}\times\Upsilon)\cup(\mathbb{R}_+\times\{\gamma_0\})) = 0.
\end{equation*}
For $i\in\mathbb{N}\setminus\{0,1\}$, $\mu_i : \Omega^{\Upsilon}_i\times \mathcal{U}_{ad}([0,T];U)\times \mathcal{B}(Y^{\Upsilon}) \rightarrow [0,1]$ are \textit{controlled transition measures} satisfying:

\begin{enumerate}
\item $\mu_i(\cdot;\Sigma)$ is measurable for each $\Sigma \in \mathcal{B}(Y^{\Upsilon})$,
\item $\mu_i(\omega_{i-1}(\omega),\alpha(\Gamma_{i-1}(\omega));\cdot)$ is a probability measure for every $\omega\in\Omega^{\Upsilon}$ and $\alpha\in\mathcal{A}$,
\item $\mu_i(\omega_{i-1}(\omega),\alpha(\Gamma_{i-1}(\omega));(\{0\}\times\Upsilon)\cup(\mathbb{R}_+\times\{\Gamma_{i-1}(\omega)\}))=0$ for every $\omega\in\Omega^{\Upsilon}$ and $\alpha\in\mathcal{A}$,
\item $\mu_i(\omega_{i-1}(\omega),\alpha(\Gamma_{i-1}(\omega));\{(\infty,\Delta)\})=1$ if $S_k(\omega) = \infty$ or $\Gamma_k(\omega) = \Delta$  for some $k\in \{1,\dots,i-1\}$, for every $\alpha\in\mathcal{A}$.
\end{enumerate}

We need to extend the definition of $\alpha\in\mathcal{A}$ to the state $(\infty,\Delta)$ by setting $\alpha(\Delta) := u_\Delta$ where $u_\Delta$ is itself an isolated cemetery  state and $\alpha$ takes in fact values in $\mathcal{U}_{ad}([0,T];U\cup\{u_\Delta\})$.

Now for a given control strategy $\alpha\in\mathcal{A}$, $\mathbb{P}_{\gamma_0}^{\alpha}$ is the unique probability measure on $(\Omega^{\Upsilon},\mathcal{T}^{\Upsilon})$ such that for each $i\in\mathbb{N}^*$ and bounded function $f$ on $\Omega^{\Upsilon}_i$

\begin{align*}
&\int_{\Omega^{\Upsilon}}f(\omega_i(\omega))\mathbb{P}_{\gamma_0}^{\alpha}(\mathrm{d}\omega) \\
& \qquad = \int_{Y^{\Upsilon}_1}\dots\int_{Y^{\Upsilon}_i}f(y_1,\dots,y_i)\mu_i(y_1,\dots,y_{i-1},\alpha(y_{i-1});\mathrm{d}y_i)\\
& \qquad \qquad \times \mu_{i-1}(y_1,\dots,y_{i-2},\alpha(y_{i-2});\mathrm{d}y_{i-1})\dots\mu_1(\alpha(\gamma_0);\mathrm{d}y_1),
\end{align*}
with $\alpha$ depending only on the variable in $\Upsilon$ when writing "$\alpha(y_{i-1})$" , $y_{i-1}=(s_{i-1},\gamma_{i-1})$. Let's now denote by $\mathcal{F}^{\Upsilon}_{\gamma,\alpha}$ and $(\mathcal{F}_t^{\Upsilon,\gamma,\alpha})_{t\geq0}$ the completed $\sigma$-fields of $\mathcal{F}^{\Upsilon}$ and $(\mathcal{F}_t^{\Upsilon})_{t\geq0}$ with all the $\mathbb{P}^{\alpha}_{\gamma}$-null sets of $\mathcal{F}^{\Upsilon}$. We then rename the intersection of these $\sigma$-fields redefine $\mathcal{F}^{\Upsilon}$ and $(\mathcal{F}_t^{\Upsilon})_{t\geq0}$ so that we have

\begin{equation*}
\mathcal{F}^{\Upsilon} := \bigcap_{\gamma\in\Upsilon \\ \alpha \in \mathcal{A}} \mathcal{F}^{\Upsilon}_{\gamma,\alpha},
\end{equation*}

\begin{equation*}
\mathcal{F}_t^{\Upsilon} :=  \bigcap_{\gamma\in\Upsilon \\ \alpha \in \mathcal{A}} \mathcal{F}_t^{\Upsilon,\gamma,\alpha} \text{ for all } t\geq 0.
\end{equation*}

Then ($\Omega^{\Upsilon},\mathcal{F}^{\Upsilon},(\mathcal{F}_t^{\Upsilon})_{t\geq0}$) is the natural filtered space of controlled jump processes.

\paragraph{Part 2. The canonical space of \textit{càdlàg} functions with values in $\Xi$.}

Let $\Omega_{\Xi}$ be the set of right-continuous functions with left limits (\textit{càdlàg} functions), defined on $\mathbb{R}_+$ with values in $\Xi$. Analogously to what we have done in Part 1, we can construct a filtered space ($\Omega^{\Xi},\mathcal{F}^{\Xi},(\mathcal{F}_t^{\Xi})_{t\geq0}$) with coordinate process $(x_t^{\Xi})_{t\geq0}$ and a probability $\mathbb{P}^{\alpha}$ on ($\Omega^{\Xi},\mathcal{F}^{\Xi}$) for every control strategy $\alpha\in\mathcal{A}$ such that the infinite-dimensional PDMP is a $\mathbb{P}^{\alpha}$-strong Markov process. For $(t,y)\in\mathbb{R}_+\times\Omega_{\Xi}, x_t^{\Xi}(y)=y(t)$.

We start with the definition of $\mathcal{F}_t^{\Xi,0} := \sigma\{x_s^{\Xi} | s\leq t\}$  for $t\in\mathbb{R}_+$ and $\mathcal{F}^{\Xi,0}:=\vee_{t\geq 0} \mathcal{F}_t^{\Xi,0} $. In Davis \cite{DavisBook} p 59, the construction of the PDMP is conducted on the Hilbert cube, the space of sequences of independent and uniformly distributed random variables in $[0,1]$. In the case of controlled PDMP, the survival function $F(t,x)$ in \cite{DavisBook} is replaced by the extension to $\xi^{\alpha}$ of $\chi^{\alpha}$ defined in Definition \ref{ControlStrat} and the construction depends on the chosen control. This extension is defined for $x:=(v,d,\tau,h,\nu)\in\Xi^{\alpha}$ by 

\[
\chi^{\alpha}_t(x) := \exp\left(-\int_0^t\lambda_d(\phi_s^{\alpha}(x),\alpha_{\tau+s}(\nu,d,h))\mathrm{d}s\right),
\]
such that for $z:=(v,d,h)\in\Upsilon$, $\chi_t^{\alpha}(z) = \chi_t^{\alpha}(v,d,0,h,v)$.

This procedure thus provides for each control $\alpha\in\mathcal{A}$ and starting point $x\in\Xi^{\alpha}$ a measurable mapping $\psi_x^{\alpha}$ from the Hilbert cube to $\Omega_{\Xi}$. Let $\mathbb{P}_x^{\alpha}:=\mathbb{P}\left[(\psi_x^{\alpha})^{-1}\right]$ denote the image measure of the Hilbert cube probability $\mathbb{P}$ under $\psi_x^{\alpha}$. 
Now for $x\in\Xi^\alpha$, let $\mathcal{F}_t^{x,\alpha}$ be the completion of $\mathcal{F}_t^{\Xi,0} $ with all $\mathbb{P}_x^{\alpha}$-null sets of $\mathcal{F}^{\Xi,0}$, and define
\begin{equation}
\mathcal{F}^{\Xi}_t := \bigcap_{ \alpha\in\mathcal{A}\\ ,x\in\Xi^\alpha} \mathcal{F}_t^{x,\alpha}.
\end{equation}

The right-continuity of $(\mathcal{F}^{\Xi}_t)_{t\geq0}$ follows from the right-continuity of $(\mathcal{F}_t^{\Upsilon})_{t\geq 0}$ and the one-to-one correspondence. The right-continuity of $(\mathcal{F}_t^{\Upsilon})_{t\geq 0}$  is a classical result on right-constant processes. For theses reasons, we lose the superscripts $\Xi$ and $\Upsilon$ consider the natural filtration $(\mathcal{F}_t)_{t\geq 0}$ in the sequel.

Now that we have a filtered probability space that satisfies the \textit{usual conditions}, let us show that the simple Markov property holds for $(X_t^\alpha)$. Let $\alpha\in\mathcal{A}$ be a control strategy, $s>0$ and $k\in\mathbb{N^*}$. By construction of the process $(X_t^{\alpha})_{t\geq0}$, 

\begin{align*}
\mathbb{P}^{\alpha}[T_{k+1}-T_k>s|\mathcal{F}_{T_k}]& = \exp\left(-\int_0^s \lambda_{d_{T_k}}(\phi_t^{\alpha}(X_{T_k}^{\alpha}),\alpha_u(\nu_{T_k},d_{T_k},h_{T_k}))\mathrm{d}u\right)\\
& = \chi^\alpha_s(X_{T_k}^{\alpha}).
\end{align*}
Now for $x\in\Xi^\alpha$, $(t,s)\in\mathbb{R}_+^2$ and $k\in\mathbb{N^*}$,

\begin{align*}
\mathbb{P}_x^{\alpha}&[T_{k+1}>t+s|\mathcal{F}_t]\mathbf{1}_{\{T_k\leq t<T_{k+1}\}}\\
& = \mathbb{P}_x^{\alpha}[T_{k+1}-T_{k}>t+s-T_k|\mathcal{F}_t]\mathbf{1}_{\{0\leq t-T_k<T_{k+1}-T_k\}}\\
& = \exp\left(-\int_{t-T_k}^{t+s-T_k} \lambda_{d_{T_k}}(\phi_u^{\alpha}(X_{T_k}^{\alpha}),\alpha_u(\nu_{T_k},d_{T_k},h_{T_k}))\mathrm{d}u\right)\mathbf{1}_{\{0\leq t-T_k<T_{k+1}-T_k\}} \quad (*)\\
& = \exp\left(-\int_{0}^{s} \lambda_{d_{T_k}}(\phi_{u+t-T_k}^{\alpha}(X_{T_k}^{\alpha}),\alpha_{u+t-T_k}(\nu_{T_k}d_{T_k},h_{T_k}))\mathrm{d}u\right)\mathbf{1}_{\{0\leq t-T_k<T_{k+1}-T_k\}}.
\end{align*}

The equality (*) is the classical formula for jump processes (see Jacod \cite{Jacod}). On the other hand, 

\begin{align*}
\chi^{\alpha}_s(X_t^{\alpha})\mathbf{1}_{\{T_k\leq t<T_{k+1}\}} & = \exp\left(-\int_0^s \lambda_{d_{t}}\Big(\phi_u^{\alpha}(X_{t}^{\alpha}),\alpha_{u+\tau_t}(\nu_{t},d_{t},h_t)\Big)\mathrm{d}u\right)\mathbf{1}_{\{T_k\leq t<T_{k+1}\}}\\ 
& =  \exp\left(-\int_0^s \lambda_{d_{T_k}}\Big(\phi_u^{\alpha}(X_{t}^{\alpha}),\alpha_{u+t-T_k}(\nu_{T_k},d_{T_k},h_{T_k})\Big)\mathrm{d}u\right)\mathbf{1}_{\{T_k\leq t<T_{k+1}\}}\\
& =   \exp\left(-\int_0^s \lambda_{d_{T_k}}\Big(\phi_{u+t-T_k}^{\alpha}(X_{T_k}^{\alpha}),\alpha_{u+t-T_k}(\nu_{T_k}d_{T_k},h_{T_k})\Big)\mathrm{d}u\right)\\
& \qquad \qquad \mathbf{1}_{\{T_k\leq t<T_{k+1}\}},
\end{align*}

because $X_t^{\alpha} = \Big(\phi_{t-T_k}^{\alpha}(X_{T_k}^{\alpha}),d_{T_k},t-T_k,h_{T_k},\nu_{T_k}\Big)$ and by the flow property $\phi_u^{\alpha}(X_{t}^{\alpha})=\phi_{u+t-T_k}^{\alpha}(X_{T_k}^{\alpha})$ on  $\mathbf{1}_{\{T_k\leq t<T_{k+1}\}}$.

Thus we showed that for all $x\in\Xi^{\alpha}$, $(t,s)\in\mathbb{R}_+^2$ and $k\in\mathbb{N^*}$,

\[
\mathbb{P}_x^{\alpha}[T_{k+1}>t+s|\mathcal{F}_t]\mathbf{1}_{\{T_k\leq t<T_{k+1}\}} = \chi^{\alpha}_s(X_t^{\alpha})\mathbf{1}_{\{T_k\leq t<T_{k+1}\}}.
\]

Now if we write $T_t^{\alpha} := \inf\{s>t : X_s^{\alpha} \neq X_{s^-}^{\alpha}\}$ the next jump time of the process after $t$, we get

\begin{equation}\label{simpleMarkov}
\mathbb{P}_x^{\alpha}[T_t^{\alpha}>t+s|\mathcal{F}_t] = \chi^{\alpha}_s(X_t^{\alpha}),
\end{equation}
which means that, conditionally to $\mathcal{F}_t$, the next jump has the same distribution as the first jump of the process started at $X_t^{\alpha}$. Since the location of the jump only depends on the position at the jump time, and not before, equality (\ref{simpleMarkov}) is just what we need to prove our process verifies the simple Markov property.

To extend the proof to the strong Markov property, the application of Theorem (25.5) (Davis \cite{DavisBook}) on the characterization of jump process stopping times on Borel spaces is straightforward.

From the results of \cite{Riedler}, there is no difficulty in finding the expression of the extended generator $\mathcal{G}^{\alpha}$ and its domain:
\begin{itemize}
\item Let $\alpha\in\mathcal{A}$. The domain $D(\mathcal{G}^{\alpha})$ of $\mathcal{G}^{\alpha}$ is the set of all measurable $f: \Xi\rightarrow \mathbb{R}$ such that 
$t\mapsto f(\phi_t^{\alpha}(x),d,\tau+t,h,\nu)$
\\
 (resp. $(v_0,d_0,\tau_0,h_0,\nu_0,t,\omega)\mapsto f(v_0,d_0,\tau_0,h_0,\nu_0)-f(v(t^-,\omega),d(t^-,\omega),\tau(t^-,\omega),\\ h(t^-,\omega),\nu(t^-,\omega))$) is absolutely continuous on $\mathbb{R}_+$ for all $x=(v,d,\tau,h,\nu)\in \Xi^{\alpha}$ (resp. a valid integrand for the associated random jump measure). 


\item 
Let $f$ be continuously differentiable w.r.t.  $v\in V$ and $\tau\in\mathbb{R}_+$. Define $h_v$ as the unique element of $V^*$ such that 
\begin{equation*}
\frac{\mathrm{d}f}{\mathrm{d}v} [v,d,\tau,h,\nu](y) = \langle h_v(v,d,\tau,h,\nu),y\rangle_{V^*,V} \qquad \forall y\in V,
\end{equation*} 
where $\frac{\mathrm{d}f}{\mathrm{d}v} [v,d,\tau,h,\nu]$ denotes the Fréchet-derivative of $f$ w.r.t $v\in E$ evaluated at $(v,d,\tau,h,\nu)$. If $h_v(v,d,\tau,h,\nu)\in V^*$ whenever $v\in V$ and is bounded in $V$ for bounded arguments then for almost every $t\in [0,T]$,
\begin{align}\label{extendedGenerator}
\mathcal{G}^{\alpha}f(v,d,\tau,h,\nu) & =   \frac{\partial}{\partial \tau}f(v,d,\tau,h,\nu)+\langle  h_v(v,d,\tau,h,\nu),Lv + f_d(v,\alpha_{\tau}(\nu,d,h)) \rangle_{V^*,V}\\
 & + \lambda_{d}(v,\alpha_{\tau}(\nu,d,h))\int_{D}[f(v,p,0,h+\tau,v)-f(v,d,\tau,h,\nu)]\mathcal{Q}_{\alpha}(\mathrm{d}p|v,d)\nonumber.
\end{align}
\end{itemize}

The bound on the continuous component of the PDMP comes from the following estimation. Let $\alpha\in\mathcal{A}$ and $x:=(v,d,\tau,h,\nu)\in\Xi^{\alpha}$ and denote by $v^{\alpha}$ the first component of $X^{\alpha}$. Then for $t\in[0,T]$,

\begin{align}
||v_t^{\alpha}||_H & \leq ||S(t)v||_H + \int_0^t ||S(t-s)f_{d_s}(v_s^{\alpha},\alpha_{\tau_s}(\nu_s,d_s,h_s))||_H\mathrm{d}s\nonumber\\
& \leq M_S||v||_H + \int_0^t M_S(b_1+b_2||v_s^{\alpha}||_H)\mathrm{d}s\label{phiBound}\\
& \leq M_S(||v||_H+b_1T)e^{M_Sb_2T}\nonumber,
\end{align}			

by Gronwall's inequality.

\section{Proof of Lemma \ref{CONTINUITYLEMMA}}\label{proofContinuity}

\paragraph{Part 1.} Let's first look at the case when $w$ is bounded by a constant $w_{\infty}$ and define for $(z,\gamma) \in \Upsilon\times \mathcal{R}$

\[
W(z,\gamma)= \int_0^{T-h} \chi_s^{\gamma}(z)\left( \int_Z w(\phi_s^{\gamma}(z),d,h+s,u)\gamma(s)(\mathrm{d}u)\right)\mathrm{d}s
\]
Now take $(z,\gamma) \in \Upsilon\times \mathcal{R}$ and suppose $(z_n,\gamma_n)\rightarrow (z,\gamma)$. Let's write $z=(v,d,h)$ and $z_n = (v_n,d_n,h_n)$ for $n\in\mathbb{N}$. For $s\in[0,T]$, let $w_n(s,u):=w(\phi_s^{\gamma_n}(z_n),d_n,h_n+s,u)$ and $w(s,u):=w(\phi_s^{\gamma}(z),d,h+s,u)$. Let also $a_n = \min(T-h,T-h_n)$ and $b_n = \max(T-h,T-h_n)$. 
\clearpage
\noindent Then 
\begin{align*}
|W(z_n,\gamma_n)-W(z,\mu)| &\leq \left|\int_{a_n}^{b_n}  \chi_s^{\gamma_n}(z_n)\int_Z w_n(s,u) \gamma_n(s)^n(\mathrm{d}u)\mathrm{d}s\right|\\
& \qquad +  \int_0^{T-h} \chi_s^{\gamma_n}(z_n) \int_Z |w_n(s,u)-w(s,u)|  \gamma_n(s)(\mathrm{d}u)\mathrm{d}s\\
& \qquad + \Big|\int_0^{T-h} \chi_s^{\gamma_n}(z_n)  \int_Z w(s,u) \gamma_n(s)(\mathrm{d}u)\mathrm{d}s\\
& \qquad\qquad\qquad -\int_0^{T-h} \chi_s^{\gamma}(z) \int_Z w(s,u)  \gamma_n(s)(\mathrm{d}u)\mathrm{d}s\Big|\\
& \qquad + \Big|\int_0^{T-h} \chi_s^{\gamma}(z) \int_Z w(s,u)  \gamma_n(s)(\mathrm{d}u)\mathrm{d}s\\
& \qquad\qquad\qquad -\int_0^{T-h}\chi_s^{\gamma}(z) \int_Z w(s,u)  \gamma_n(s)(\mathrm{d}u)\mathrm{d}s\Big|
\end{align*}
The first term on the right-hand side converges to zero for $n\rightarrow \infty$ since the integrand is bounded.
\begin{align*}
 \int_0^{T-h} \chi_s^{\gamma_n}(z_n) \int_Z |w_n(s,u)-w(s,u)|  \gamma_n(s)(\mathrm{d}u)\mathrm{d}s & \leq \int_0^{T-h} e^{-\delta s} \sup_{u\in U}|w_n(s,u)-w(s,u)|\mathrm{d}s\\
&\qquad \xrightarrow[n\rightarrow \infty ]{} 0 \\
\end{align*}
by dominated convergence and the continuity of $w$ and of $\phi$ proved in Lemma \ref{ContinuityTheo}.
\begin{align*}
\left|\int_0^{T-h}\left(\chi_s^{\gamma_n}(z_n)-\chi_s^{\gamma}(z)\right) \int_Z w(s,u)  \mu^n_s(\mathrm{d}u)\mathrm{d}s\right| & \leq w_{\infty} \int_0^{T-h}\left|\chi_s^{\gamma_n}(z_n)-\chi_s^{\gamma}(z)\right|\mathrm{d}s\\
& \qquad \xrightarrow[n\rightarrow \infty ]{} 0 \\
\end{align*}
again by dominated convergence, provided that for $s\in [0,T]$, the convergence \\ $ \chi_s^{\gamma_n}(z^n) \xrightarrow[n\rightarrow \infty]{} \chi_s^{\gamma}(z)$ holds. For this convergence to hold it is enough that for $t\in [0,T],$
\[
\int_0^t \int_Z \lambda_{d_n}(\phi_s^{\gamma_n}(z_n),u)\gamma_n(s)(\mathrm{d}u)\mathrm{d}s  \xrightarrow[n\rightarrow \infty ]{} \int_0^t \int_Z \lambda_{d}(\phi_s^{\gamma}(z),u)\gamma(s)(\mathrm{d}u)\mathrm{d}s.
\]
It is enough to take $n$ large enough so that $d_n=d$ and to write
\begin{align*}
& \int_0^t \left( \int_Z \lambda_{d}(\phi_s^{\gamma_n}(z_n),u)\gamma_n(s)(\mathrm{d}u) - \int_Z \lambda_{d}(\phi_s^{\gamma}(z),u)\gamma(s)(\mathrm{d}u)\right)\mathrm{d}s  =\\
&\qquad \int_0^t \int_Z \left(  \lambda_{d}(\phi_s^{\gamma_n}(z_n),u)-\lambda_{d}(\phi_s^{\gamma}(z),u)\right) \gamma_n(s)(\mathrm{d}u) \mathrm{d}s\\
&\qquad + \int_0^t \left(\int_Z\lambda_{d}(\phi_s^{\gamma}(z),u)\gamma_n(s)(\mathrm{d}u) - \int_Z \lambda_{d}(\phi_s^{\gamma}(z),u)\gamma(s)(\mathrm{d}u) \right)\mathrm{d}s
\end{align*}
By the local Lipschitz property of $\lambda_{d}$, 
\begin{equation*}
\left|\int_0^t \int_Z \left(  \lambda_{d}(\phi_s^{\gamma_n}(z_n),u)-\lambda_{d}(\phi_s^{\gamma}(z),u)\right) \gamma_n(s)(\mathrm{d}u) \mathrm{d}s\right| \leq l_{\lambda} \int_0^t ||\phi_s^{\gamma_n}(z_n) - \phi_s^{\gamma}(z)||_H\mathrm{d}s
\end{equation*}
and $\int_0^t ||\phi_s^{\gamma_n}(z_n) - \phi_s^{\gamma}(z)||_E\mathrm{d}s \leq t\sup_{s\in[0,T]} || \phi_s^{\gamma_n}(z_n) - \phi_s^{\gamma}(z)||_H \xrightarrow[n\rightarrow \infty]{} 0$ by Lemma \ref{ContinuityTheo}. The second term converges to zero by the definition of the weakly* convergence in $L^{\infty}(M(Z))$.

\medskip

\paragraph{Part 2.} In the general case where $|w|\leq w_c B^{*}$, let $w^B(z,u) = w(z,u)-c_wB^*(z)\leq 0$ for $(z,u)\in \Upsilon\times U$. $w^B$ is a continuous function and there exists a nonincreasing sequence $(w_n^B)$ of bounded continuous functions such that $w_n^B  \xrightarrow[n\rightarrow \infty]{}  w^B$. By the first part of the proof we know that 
\[
W_n(z,\gamma) = \int_0^{T-h} \chi_s^{\gamma}(z)\int_Z w_n^B(\phi_s^{\gamma}(z),d,h+s,u)\mu_s(\mathrm{d}u)\mathrm{d}s
\]
is bounded, continuous, decreasing and converges to 
\[
W(z,\gamma) - c_w\int_0^{T-h} \chi_s^{\gamma}(z)b( \phi_s^{\gamma}(z))e^{\zeta^*(T-h-s)}\mathrm{d}s
\]
which is thus upper semicontinuous. Since $b$ is a continuous bounding function it is easy to show that 
\[
(z,\gamma) \rightarrow \int_0^{T-h} \chi_s^{\gamma}(z)b( \phi_s^{\gamma}(z))e^{\zeta^*(T-h-s)}\mathrm{d}s
\]
is continuous so that in fact $W$ is upper semicontinuous. 
Now considering the function $w_B(z,u) = -w(z,u)-c_wB^*(z)\leq 0$ we easily show that $W$ is also lower semicontinuous so that finally $W$ is continuous.

Now the continuity of the applications $(z,\gamma)\rightarrow c'(z,\gamma)$ and $(z,\gamma)\rightarrow (\mathcal{Q}'w)(z,\gamma)$ comes from the previous result applied to the continuous functions defined for $(z,u) \in \Upsilon\times U$ by $w_1(z,u) := c(v,u)$ and $w_2(z,u):=\lambda_d(v,u)\int_D w(v,r,h)\mathcal{Q}(\mathrm{d}r|v,d,u)$ with $z=(v,d,h)$. Here the different assumptions of continuity \ref{hyp:lambda}2.3., \ref{hyp:cost}1. and \ref{hyp:Q} are needed.

\end{appendices}

\nocite{*}
\bibliographystyle{plain}
\bibliography{Preprint_PDMP}

\end{document}